\documentclass[12pt]{article}

\usepackage{forest}
\usepackage{amsthm, amssymb}
\usepackage{amsmath}               % great math stuff
\usepackage{amsfonts}              % for blackboard bold, etc
\usepackage{amsthm}                % better theorem environments
\usepackage{mathrsfs}

\usepackage{verbatim} %this package enables us to create multiline comments
\usepackage{xspace}
\usepackage{graphicx}  % to include figures
\usepackage[small]{caption} %for fig captions, etc.
\usepackage{color}
\newcommand{\breakingcomma}{%
  \begingroup\lccode`~=`,
  \lowercase{\endgroup\expandafter\def\expandafter~\expandafter{~\penalty0 }}}
\usepackage{float}
\usepackage{wrapfig}

\usepackage{titlesec}
\titleformat{\section}{\large\bfseries}{\thesection}{.5em}{}

\newtheorem{theorem}{Theorem}[section]
\newtheorem{lemma}[theorem]{Lemma}
\newtheorem{remark}[theorem]{Remark}
\newtheorem{proposition}[theorem]{Proposition}
\newtheorem{corollary}[theorem]{Corollary}
\newtheorem{definition}[theorem]{Definition}
\newtheorem{example}[theorem]{Example}
\newcount\refno
\newcommand\be{\begin{equation}}
\newcommand\ee{\end{equation}}
\newcommand\bn{\begin{eqnarray}}
\newcommand\en{\end{eqnarray}}
\newcommand\bns{\begin{eqnarray*}}
\newcommand\ens{\end{eqnarray*}}
\newcommand\bd{\begin{definition}}
\newcommand\ed{\end{definition}}
\newcommand\br{\begin{remark}}
\newcommand\er{\end{remark}}
\newcommand\bt{\begin{theorem}}
\newcommand\et{\end{theorem}}
\newcommand\bp{\begin{proposition}}
\newcommand\ep{\end{proposition}}
\newcommand\bc{\begin{corollary}}
\newcommand\ec{\end{corollary}}
\newcommand\bl{\begin{lemma}}
\newcommand\el{\end{lemma}}

\newcommand\bR{{\mathbb R}}
\newcommand\bN{{\mathbb N}}

\newcommand\cA{{\cal A}}
\newcommand\cB{{\cal B}}
\newcommand\cD{{\cal D}}
\newcommand\cE{{\cal E}}
\newcommand\cR{{\cal R}}

\newcommand\cN{{\cal N}}

\newcommand\cS{{\cal S}}

\newcommand{\N}{\mbox{$\mathbb{N}$}}
\newcommand{\NS}{\mbox{\scriptsize${\mathbb{N}}$}}

\newcommand{\F}{\mbox{$\mathcal{F}$}}

\allowdisplaybreaks

\begin{document}

\title{The Double Almost-Riordan Arrays and Their Sequence Characterization, Compression, and Total Positivity}

\author{Tian-Xiao He 
\\
{\small Department of Mathematics}\\
 {\small Illinois Wesleyan University}\\
 {\small Bloomington, IL 61702-2900, USA}\\
}

\date{}

%%%%%%%%%%%%%%%%%%%%%%%%%%%%%%%%%%%%%%%%%%%%%%%%%%%%%%%%%%%%%%%%%%%%%%
\maketitle
\setcounter{page}{1}
\pagestyle{myheadings}
\markboth{T. X. He }
%and J. Zheng}
{Double Almost-Riordan Arrays and their sequence characterizations}

%%%%%%%%%%%%%%%%%%%%%%%%%%%%%%%%%%%%%%%%%%%%%%%%%%%%%%%%%%%%%%%%%%%%%%
\begin{abstract}
\noindent 
In this paper, we define double almost-Riordan arrays and find that the set of all double almost-Riordan arrays forms a group, called the double almost-Riordan group. We also obtain the sequence characteristics of double almost-Riordan arrays and give the production matrices for double almost-Riordan arrays. We define the compression of double almost-Riordan arrays and present their sequence characterization. Finally we give a characteristic for the total positivity of double Riordan arrays, by using which we discuss the total positivity for compressed double almost-Riordan arrays.



\end{abstract}

\section{Introduction}

Riordan arrays are infinite, lower triangular matrices defined by the generating function of their columns. They form a group, denoted by $\mathcal{R}$ and called {\em the Riordan group} (see Shapiro, Getu, W. J. Woan and L. Woodson \cite{SGWW}). 

A Riordan array $(g(t),f(t))$ is a pair of formal power series $g(t) = \sum_{n\geq 0}g_nt^n$ and $f(t) = \sum_{n\geq 1}f_nt^n$, with $g_0\not= 0$ and $f_1\not= 0$. It defines an infinite lower triangular array $[dn,k]_{n,k\geq 0}$ according to the rule $d_{n,k} = [t^n ]g(t)f(t)^k$. The set of all Riordan arrays forms a group under matrix multiplication

\[
(g(t), f (t))(h(t), l(t)) = (g(t)h(f (t)), l(f (t))).
\]
An almost-Riordan array is defined by an ordered triple $(a|g, f )$ of power series where $a(t) =\sum_{n\geq 0} a_nt^n$, $g(t) = \sum_{n\geq 0} g_nt^n$ and $f(t) = \sum_{n\geq1} f_nt^n$, with $a_0, g_0, f_1\not= 0$. The array is identified with the lower-triangular matrix defined as follows: its first column is given by the expansion of $a(t)$. The remaining element of the infinite tridiagonal matrix coincide with the Riordan array $(g(t), f (t))$.

An infinite lower triangular matrix $D = [d_{n,k}]_{n,k\geq 0}$ is called double Riordan array, if $g(t)$ gives column zero and $f_1(t)$ and $f_2(t)$ are multiplier functions, where $g(t) =\sum_{n\geq 0} g_{2n}t^{2n}$, $f_1(t) =\sum_{n\geq 0} f_{1,2n+1}t^{2n+1}$ and $f_2(t) =\sum_{n\geq 0} f_{2,2n+1}t^{2n+1}$ together with $g_0,f_{1,1}, f_{2,1}\not= 0$. Then the double Riordan array related to $g(t), f_1(t)$ and $f_2(t)$, denoted by $(g; f_1, f_2)$, has the column vector $(g; gf_1, gf_1f_2, gf_1^2f_2, gf_1^2f_2^2, \ldots)$.

Riordan arrays, almost-Riordan arrays and double Riordan arrays have emerged as a powerful tool with broad applications in various branches of mathematics. With their intricate connections to combinatorics, group theory, matrix theory, and number theory, these arrays serve as a bridge between these disciplines. This paper presents the study of the double almost-Riordan arrays and the double almost-Riordan group defined in \cite{He24}. Specifically, it studies the sequence characterization, compression, and total positivity of double almost-Riordan arrays.

More formally, let us consider the set of all formal power series (f.p.s.) in $t$, $\F = {\mathbb K}[\![$$t$$]\!]$, with a field ${\mathbb K}$ of characteristic $0$ (e.g., ${\mathbb Q}$, ${\mathbb R}$, ${\mathbb C}$, etc.). The \emph{order} of $f(t)  \in \F$, $f(t) =\sum_{k=0}^\infty f_kt^k$ ($f_k\in {\bR}$), is the minimal number $r\in\N$ such that $f_r \neq 0$. Denote by $\F_r$ the set of formal power series of order $r$. Let $g(t) \in \F_0$ and $f(t) \in \F_1$. Then, the pair $(g(t) ,\,f(t) )$ defines the {\em (proper) Riordan array} $D=(d_{n,k})_{n,k\in \NS}=(g(t), f(t))$ having
  
\begin{equation}\label{Radef}
d_{n,k} = [t^n]g(t) f(t) ^k
\end{equation}
or, in other words, having $g(t) f(t)^k$ as the generating function whose coefficients make-up the entries of column $k$. 

From the {\it fundamental theorem of Riordan arrays} (see \cite{Sha1}), it is immediate to show that the usual row-by-column product of two Riordan arrays is also a Riordan array:
\begin{equation}\label{Proddef}
    (g_1(t) ,\,f_1(t) )  (g_2(t) ,\,f_2(t) ) = (g_1(t) g_2(f_1(t) ),\,f_2(f_1(t) )).
\end{equation}
The Riordan array $I = (1,\,t)$ acts as an identity for this product. Thus, the set of all Riordan arrays forms the Riordan group $\mathcal{R}$.

Several subgroups of $\mathcal{R}$ are important and have been considered in the literature:
\begin{itemize} \item the set $\mathcal{A}$ of {\em Appell arrays} is the collection of all Riordan arrays $R = (g(t) ,\,t )$ in ${\cR}$; 
\item the set $\mathcal{L}$ of {\em Lagrange arrays} is the collection of all Riordan arrays $R = (1 ,\,f(t) )$ in ${\cR}$;
\item the set $\mathcal{B}$ of \emph{Bell} or {\em renewal arrays} is the collection of all Riordan arrays $R = (g(t) ,\,t g(t))$ in ${\cR}$;
\item the set $\mathcal{H}$ of \emph{hitting-time arrays} is the collection of all Riordan arrays $R = (tf'(t)/f(t) ,\,f(t))$ in ${\cR}$;
\item the set $\mathcal{D}$ of the Riordan arrays $R = (f'(t) ,\, f(t) )$ in ${\cR}$ is called the {\it derivative subgroup}.
\item the set $\mathcal{E}$ of the Riordan arrays $R=((f(t)/t)^{r}f'(t)^{s}, f(t))$ in ${\cR}$ with real or complex $r$ and $s$ is called {\it Lu\'zon-Merlini-Mor\'on-Sprugnoli (LMMS) subgroup}, denoted by ${\cE}[r,s]$, which includes ${\cD}={\cE}[0,1]$ as its special case (see \cite{LMMS2}). 
\end{itemize}

From \cite{Rog}, an infinite lower triangular array $[d_{n,k}]_{n,k\in{\bN}}=(g(t), f(t))$ is a Riordan array if and only if an {\it $A$-sequence} $A=(a_0\not= 0, a_1, a_2,\ldots)$ exists such that for every $n,k\in{\bN}$ there holds 
\be\label{eq:1.1}
d_{n+1,k+1} =a_0 d_{n,k}+a_1d_{n,k+1}+\cdots +a_nd_{n,n},
\ee 
which is shown in \cite{HS} to be equivalent to 
\be\label{eq:1.2}
f(t)=tA(f(t)).
\ee
Here, $A(t)$ is the generating function of the $A$-sequence. In \cite{MRSV} it is also shown that a unique {\it $Z$-sequence} $Z=(z_0, z_1,z_2,\ldots)$ exists such that every element in column $0$ can be expressed as the linear combination 
\be\label{eq:1.3}
d_{n+1,0}=z_0 d_{n,0}+z_1d_{n,1}+\cdots +z_n d_{n,n},
\ee
or equivalently (see \cite{HS}),
\be\label{eq:1.4}
g(t)=\frac{d_{0,0}}{1-tZ(f(t))}.
\ee

Denote the {\it upper shift matrix} by $U$, i.e., 
\[
U=(\delta_{i+1,j})_{i,j\geq 0}=\left [ \begin{array}{llllll} 0& 1& 0& 0& 0& \cdots\\
0&0 &1& 0& 0&\cdots\\
0 &0& 0&1& 0& \cdots\\
0&0& 0& 0&1&\cdots\\
\vdots &\vdots& \vdots& \vdots&\vdots&\ddots\end{array}\right]
\]
and 
\be\label{eq:1.4-2}
P=\left [ \begin{array}{llllll} z_0& a_0& 0& 0& 0& \cdots\\
z_1& a_1 &a_0& 0& 0&\cdots\\
z_2 &a_2& a_1& a_0& 0& \cdots\\
z_3& a_3& a_2& a_1&a_0&\cdots\\
\vdots &\vdots& \vdots& \vdots&\vdots&\ddots\end{array}\right]=\left( Z(t), A(t), tA(t), t^{2}A(t),\ldots\right), 
\ee
where the rightmost expression is the representation of $P$ by using its column generating functions. Here, $P$ is called the {\it production matrix} or {\it $P$-matrix characterization} or simply {\it $P$ matrix} (see Deutsch, Ferrari, and Rinaldi 
\cite{DFR05, DFR}). From \cite{DFR05, DFR} or Proposition 2.7 of \cite{He15}, the $P$-matrix of Riordan array $R$ satisfies 

\be\label{1.5}
P =R^{-1}UR=R^{-1}\overline{R},
\ee
where $\overline{R}$ is the truncated Riordan array $R$ with the first row omitted.

The ECO method is a constructive method to produce all the objects of a given class, according to the growth of a certain parameter (in terms of the size) of the objects. A complete description of the ECO method and its applications for the enumeration of several classes of combinatorial objects is given in \cite{BDLPP}. The roots of the ECO method can be traced back to the paper \cite{CGHK}. If an ECO construction is sufficiently regular, then it is often possible to describe it using a succession rule, whose definition is first introduced by Julian West in \cite{West}. Some algebraic properties of succession rules have been determined in \cite{FPPR}. 

A {\it succession rule} is a formal system consisting of an axiom $(a)$, $a\in {\bN}_+$, and a set of productions 
\[
\left\{ (k_j)\to (e_1(k_j))(e_2(k_j))\cdots (e_j(k_j)): j\in{\bN}\right\},
\]
where $e_i:{\bN}_+\to {\bN}_+$, for deriving the successors $(e_1(k))(e_2(k))\cdots (e_k(k))$ of any given label $(k)$, $k\in{\bN}$. Thus, we may write a succession rule, denoted by $\Omega$, as 

\be\label{eq:1.5}
\Omega:\left\{ \begin{array}{ll} (a)\\ (k)\to  (e_1(k))(e_2(k))\cdots (e_k(k)).\end{array}\right.
\ee
$(a), (k), (k),$ and $(e_i(k))$ ($a,k,e_i(k)\in{\bN}_+$) are called the labels of $\Omega$. The succession rule can be presented as a {\it generating rooted tree} called a {\it generating tree}, whose vertices are the labels of $\Omega$, where $(a)$ is the label of the root, and each node labeled $(k)$ has $k$ sons labeled by $(e_1(k)), \dots,$ $(e_k(k))$, respectively. Denote by $f_n$ the number of nodes at level $n$ in the generating tree determine by $\Omega$. We call $\{ f_n\}_{n\geq 0}$ the sequence $\{ f_n\}_{n\geq 0}$ associated with $\Omega$. In \cite{DFR05} the authors represent the succession rule $\Omega$ by an infinite matrix $P=(p_{i,j})_{i,j\geq 0}$, called the {\it production matrix} of the succession rule $\Omega$, where $p_{i,j}$ is the number of labels $\ell_j$ produced by label $\ell_i$, and $\{(\ell_k)\}_{k\geq 0}$ is the label set of $\Omega$ with the axiom label $\ell_0$. Simply speaking, if $\Omega$ containing root $(a)$ satisfies succession rule $(a)\to (a)\dots (a) (a+1)\dots (a+1)\dots (a+m)\dots (a+m)$, which we shorten to $(a)^{n_1}(a+1)^{n_2}\dots (a+m)^{n_m}$, then the first row of $P$ is $(n_1, n_2, \cdots, n_m, 0,\cdots)$, where $n_j$ ($j=1,2,\ldots, m$) may be zero, say $n_k=0$, if $(a+k)$ is not a label. Similarly, we may write other rows of $P$ row by row. 

In \cite{DFR} the following result is proved, which can be viewed as a particular case of \eqref{1.5} when $A$- and $Z$- sequences are non-negative sequences. The case of not non-negative $A$- and/or $Z$-sequences is discussed in \cite{He15}. 

\begin{proposition}\label{pro:1.2} \cite{DFR}
Let $P$ be an infinite production matrix and let $A_P$ be the ECO matrix induced by $P$ defined before. Then $A_P$ is a Riordan matrix if and only if $P$ is of the form 

\be\label{eq:1.8}
P=\left [ \begin{array}{llllll} z_0& a_0& 0& 0& 0& \cdots\\
z_1& a_1 &a_0& 0& 0&\cdots\\
z_2 &a_2& a_1& a_0& 0& \cdots\\
z_3& a_3& a_2& a_1&a_0&\cdots\\
\vdots &\vdots& \vdots& \vdots&\vdots&\ddots\end{array}\right],
\ee
where $z_j, a_j\geq 0$ for $j=0,1,\ldots$.
\end{proposition}

As an example, let us consider the Fibonacci tree $F_1$ shown in Horibe \cite{Hor}, which satisfies the following succession rule:

\be\label{eq:1.5-2}
\Omega: \left\{ \begin{array}{ll} (a)\\ a(k)\to a(k) a(k+1).\end{array}\right.
\ee
Figure $1$ presents first few levels of $F_1$.
\begin{figure}[H]
\centering
\includegraphics{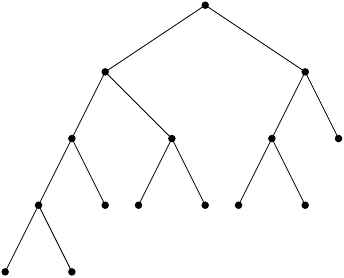}
\caption{Figure $1$. Fibonacci tree $F_1$}
\end{figure}
The $P$ matrix of the succession rule is 

\be\label{eq:1.8-2}
P=\left [ \begin{array}{llllll} 1& 1& 0& 0& 0& \cdots\\
1& 1 &1& 0& 0&\cdots\\
0&0& 1& 1& 0& \cdots\\
0& 0& 0& 1&1&\cdots\\
\vdots &\vdots& \vdots& \vdots&\vdots&\ddots\end{array}\right].
\ee

From \eqref{1.5}, one may finds the corresponding Riordan array $R$ is the pascal matrix $(1/(1-t), t/(1-t))$: 

\be\label{eq:1.8-3}
R=\left [ \begin{array}{llllll} 1& 0& 0& 0& 0& \cdots\\
1& 1 &0& 0& 0&\cdots\\
1&2& 1& 0& 0& \cdots\\
1& 3& 3& 1&0&\cdots\\
\vdots &\vdots& \vdots& \vdots&\vdots&\ddots\end{array}\right].
\ee
The sum of the numbers along the parallel lines in the direction of $(1,1)$ are the Fibonacci numbers, $\{1,1,2,3,5,8,\ldots\}$.

Succession rule can generates more wide class of lower triangle matrices that do not belong to Riordan array class.  for instance, in Stanley \cite{Sta}, the Fibonacci tree $F_2$ is defined to be the rooted tree with root $v$, such that the root has degree one, the child of every vertex of degree one has degree two, and the two children of every vertex of degree two have degrees one and two. 
The succession rule $\Omega$ for the Fibonacci tree $F_2$ is 

\be\label{eq:1.5-3}
\Omega:\left\{ \begin{array}{ll} (a_{2l})\to (a_{2l})(a_{2l+1})\\ (a_{2l+1})\to  (a_{2l+2}),\end{array}\right.
\ee
for $\ell =0,1,2,\ldots$.  This succession rule generates the following lower triangle matrix $\bar R$:

%\begin{center}
\be\label{1.5-4}
\hat R=
\begin{bmatrix}
    1&0&0&0&0&0&...\\
    1&1&0&0&0&0&...\\
    1&1&1&0&0&0&...\\
    1&1&2&1&0&0&...\\
    1&1&3&2&1&0&...\\
    1&1&4&3&3&1&...\\
%    1&1&5&4&6&3&...\\
%    1&1&6&5&10&6&...\\
    \vdots &\vdots&\vdots&\vdots&\vdots&\vdots&\ddots\\
\end{bmatrix} 
\ee
Matrix $\hat R$ is called {\it Fibonacci-Stanley array}. The row sums of $\hat R$ are Fibonacci numbers. Obviously, $\hat R$ is not a Riordan array. In \cite{He18}, we show that $\hat R$ is the compression of a  double Riordan array. For the convenience of readers, we quote this detailed content below. This is also a motivation for studying double Riordan arrays.

In construction of a Riordan array, one multiplier function $h$ is used to multiply one column to obtain the next column. Suppose alternating rules are applied to generate an infinite matrix similar to a Riordan array. To consider this case, one may use two multiplier functions, denoted by $f_1$ and $f_2$, respectively. Davenport, Shapiro, and Woodson \cite{DSW} require that $g$ to be an even function and $f_1$ and $f_2$ to be odd functions so that they can develop an analog of the fundamental theorem and thus obtain a group structure. Let $g(t)$ = $\sum _{k=0}^{\infty} g_{2k}t^{2k}$, $f_{1}(t)$ = $\sum_{k=0}^{\infty} f_{1,2k+1} t^{2k+1}$, and $f_{2} (t)$ = $\sum_{k=0}^{\infty} f_{2,2k+1} t^{2k+1}$. Then the double Riordan matrix in terms of  $g(t)$, $f_1(t)$, and $f_2(t)$, denoted by $(g; f_{1}, f_{2})$, is defined by the generating function of its columns as 

\begin{center}
($g$, $gf_{1}$, $gf_{1}f_{2}$, $gf_{1}^{2}f_{2}$, $gf_{1}^{2}f_{2}^{2}$,...).
\end{center}
There two cases of the {\it first fundamental theorem of double Riordan arrays}:

\be\label{1.6}
(g; f_{1}, f_{2})A(t)=B(t),
\ee
where for $A(t)=\sum_{k\geq 0}a_{2k}t^{2k}$ and $A(t)=\sum_{k\geq 0}a_{2k+1}t^{2k+1}$, we have $B(t)=dA(\sqrt{h_{1}h_{2}})$ and $B(t)=d\sqrt{h_{1}/h_{2}}A(\sqrt{h_{1}h_{2}})$, respectively. Based on the fundamental theorem of double Riordan arrays, we may define a multiplication of two double Riordan arrays as 

\be\label{1.7}
(g; f_{1},f_{2})(d; h_{1},h_{2})=(g d(\sqrt{f_{1}f_{2}}); \sqrt{f_{1}/f_{2}}h_{1}(\sqrt{f_{1}f_{2}}), \sqrt{f_{2}/f_{1}}h_{2}(\sqrt{f_{1}f_{2}})),
\ee
where $d(t)$ = $\sum _{k=0}^{\infty} d_{2k}t^{2k}$, $h_{1}(t)$ = $\sum_{k=0}^{\infty} h_{1,2k+1} t^{2k+1}$, and $h_{2} (t)$ = $\sum_{k=0}^{\infty} h_{2,2k+1}$ $t^{2k+1}$. The collection of all double Riordan arrays associated with the multiplication defined above form a group called double Riordan group and denoted by $\cal{D}\cal{R}$, where the identity is $(1,t,t)$.  From \cite{DSW}, ${\cS}:=\{ (g;f_1, f_2)\in {\cD}{\cR}:g=1\}$, ${\cB}_1:=\{(g;f_1, f_2)\in {\cD}{\cR}:f_1=tg\}$, ${\cB}_2:=\{(g; f_1, f_2)\in {\cD}{\cR}:f_2=tg\}$, and ${\cA}:=\{(g;t,t)\in{\cD}{\cR}\}$ are associate, type-1 Bell, type-2 Bell, and Appel subgroups of ${\cD}{\cR}$, respectively. 

In \cite{DSW}, subgroups ${\cA}:=\{(1; f_1,f_2)\in{\cD}{\cR}\}$, ${\cB}_1:=\{(g; tg,f_2)\in{\cD}{\cR}\}$, ${\cB}_2:=\{(g; f_1,tg)\in{\cD}{\cR}\}$, and ${\cN}:=\{ (g;t,t)\in{\cD}{\cR}\}$ are given. Here, ${\cN}$ is a normal subgroup of ${\cD}{\cR}$ and ${\cD}{\cR}$ is the semidirect product of ${\cN}$ and ${\cA}$. 

If $(g;f_1,f_2)$ is a double Riordan array, where $g(t)$ = $\sum _{k=0}^{\infty} g_{2k}t^{2k}$, $f_{1}(t)$ = $\sum_{k=0}^{\infty} f_{1,2k+1} t^{2k+1}$, and $f_{2} (t)$ = $\sum_{k=0}^{\infty} f_{2,2k+1} t^{2k+1}$. Then $(\hat g; \hat f_1, \hat f_2)$, where $\hat g(t)$ = $\sum _{k=0}^{\infty} g_{2k}t^{k}$, $\hat f_{1}(t)$ = $\sum_{k=0}^{\infty} f_{1,2k+1} t^{k+1}$, and $\hat f_{2} (t)$ = $\sum_{k=0}^{\infty} f_{2,2k+1} t^{k+1}$, is the compression of the double Riordan array $(g;f_1,f_2)$. It can be seen from \cite{He18} that $\hat R$ shown above is the compression of the Double Riordan array $(1/(1-t^2), t, t/(1-t^2))$. 

In \cite{He18}, the following sequence characterization of a double Riordan array in ${\cD}{\cR}$ is given. 

\begin{theorem}\label{thm:0.2}\cite{He18}
$(g;f_1,f_2)=(d_{n,k})_{n,k\geq 0}$ is in ${\cD}{\cR}$ with $d_{n,n}\not= 0$ if and only if there exist $A_1$-sequence, $\{ a_{1,0}, a_{1,1}, a_{1,2},\ldots\}$, $A_2$- sequence, $\{ a_{2,0}\not=0, a_{2,1}, a_{2,2},\ldots\}$, and $Z$-sequence, $\{z_0,z_1,z_2,\ldots\}$ such that 

\bn\label{0.1}
&&d_{n,2k-1}=a_{1,0}d_{n-1,2k-2}+a_{1,1}d_{n-1,2k}+a_{1,2}d_{n-1,2k+2}+\cdots,\quad k\geq 1\nonumber\\
&&d_{n,2k}=a_{2,0}d_{n-2,2k-2}+a_{2,1}d_{n-2,2k}+a_{2,2}d_{n-2,2k+2}+\cdots,\quad k\geq 1\nonumber\\
&&d_{n,0}=z_{0}d_{n-2,0}+z_{1}d_{n-2,2}+z_{2}d_{n-2,4}+\cdots 
\en
for $n\geq 0$. Furthermore, denoting the generating functions of $A_1$-sequence, $A_2$-sequence, and $Z$-sequence by 
$A_1(t)=\sum_{k\geq 0}a_{1,k}t^{2k}$, $A_2(t)=\sum_{k\geq 0} a_{2,k}t^{2k}$, and $Z(t)=\sum_{k\geq 0} z_kt^{2k}$, respectively, we have an equivalent form of \eqref{0.1}

\bn\label{0.2}
&&f_1=tA_1(\sqrt{f_1f_2}),\\
&&f_1f_2=t^2A_2(\sqrt{f_1f_2}),\label{0.3}\\
&&g(t)=\frac{g(0)}{1-t^2Z(\sqrt{f_1f_2})}.\label{0.4}
\en
 \end{theorem}

Some recently published results on double Riordan arrays can be found in Sun and Sun \cite{SS}, Zhang and Zhao \cite{ZZ}, etc. and their references.

We now present the definition of almost-Riordan arrays. 

\begin{definition}\label{def:1.3}\cite{Barry16}
Let $d, g\in \F_0$ with $d(0), g(0)=1$ and $f\in \F_1$ with $f'(0)=1$. 
We call the following matrix an almost-Riordan array with respect to $d,g,$ and $f$ and denote it by $(d|\,g,f)$:

\be\label{1.13}
(d|\,g,f)=(d, tg, tgf, tgf^2,\cdots),
\ee
where $d$, $tg$, $tgf$, $t^2gf\cdots$, are the generating functions of the $0$th, $1$st, $2$nd, $3$rd, $\cdots$, columns of the matrix $(a|\,g,f)$, respectively. 
\end{definition}

It is clear that $(d|\,g,f)$ can be written as 
\be\label{1.14}
(d|\,g,f)=\left( \begin{matrix} d(0) & 0\\ (d-d(0))/t & (g,f) \end{matrix}\right),
\ee
where $(g,f)=(g,gf, gf^2,\ldots)$, a Riordan array. Particularly, if $d=g$ and $f=t$, then the almost-Riordan array $(d|\,g,f)$ reduces to the Appell-type Riordan array $(g, t)$. 

Barry, Pantelidis, and the author \cite{BHP} present the following operation form for the almost-Riordan group.

\begin{theorem}\label{thm:0.3}\cite{BHP}
The set of all almost-Riordan arrays defined by \eqref{1.13} forms a group, denoted by $a{\cR}$, with respect to the multiplication defined by 
\be\label{1.15}
(a|\,g,f)(b|\,d,h)=\left( \left. a+\frac{tg}{f}(b(f)-1)\right |\,gd(f), h(f)\right),
\ee
where $a, g, b, d\in \F_0$ and $f, h\in \F_1$, and $(a|\,g,f)$ and $(b|\,d,h)$ are the almost-Riordan arrays defined by \eqref{1.13} or \eqref{1.14}.
\end{theorem}

Alp and Kocer give the sequence characterization of almost-Riordan arrays and define the exponential almost-Riordan group in \cite{AK} and \cite{AK24}, respectively. Slowik and the author discuss the total positivity of almost-Riordan arrays in \cite{HS24} and the total positivity of quasi-Riordan arrays in \cite{HS}, respectively, where the set of quasi-Riordan arrays forms a normal subgroup of the almost-Riordan group defined in \cite{He}. 

Let us introduce one more group that is closely related to the groups of Riordan and almost-Riordan arrays.

\begin{definition}\label{def:0.3}\cite{He} 
Let $g\in \F_0$ with $g(0)=1$ and $f\in \F_1$. We call the following matrix a quasi-Riordan array and denote it by $[g,f]$.
\be\label{0.6}
[g,f]:=(g,f,tf,t^2f,\ldots),
\ee
where $g$, $f$, $tf$, $t^2f\cdots$ are the generating functions of the $0$th, $1$st, $2$nd, $3$rd, $\cdots$, columns of the matrix $[g,f]$, respectively. It is clear that $[g,f]$ can be written as 
\be\label{0.8}
[g,f]=\left( \begin{matrix} g(0) & 0\\ (g-g(0))/t & (f/t,t)\end{matrix}\right),
\ee
where $(f,t)=(f,tf,t^2f,t^3f,\ldots)$ and $(f/t, t)$ is an Appell Riordan array. 
Particularly, if $f=tg$, then the quasi-Riordan array $[g,tg]=(g, t)$, a Appell-type Riordan array. 

Clearly, $[g,f]=(g|f/t, t)$. 

Let $A$ and $B$ be $m\times m$ and $n\times n$ matrices, respectively. Then we define the direct sum of $A$ and $B$ by

\be\label{0.5} 
A\oplus B =\left[ \begin{matrix} A &0 \\ 0 &B\end{matrix}\right]_{(m+n)\times(m+n)}.
\ee

In this notation the Riordan array $(g,f)$ satisfies 

\be\label{0.9}
(g,f)=[g,f]([1]\oplus (g,f)).
\ee
\end{definition}

Denote by $q{\cR}$ the set of all quasi-Riordan arrays defined by \eqref{0.6}. In \cite{He} it is shown that $q{\cR}$ is a group with respect to regular matrix multiplication. More precisely, there is the following result.

\begin{theorem}\label{thm:1.2}\cite{He}
The set of all quasi-Riordan arrays $q{\cR}$ is a group, called the quasi-Riordan group, with respect to the multiplication represented in 
\be\label{1.11}
[g,f][d,h]=\left[g+\frac{f}{t}(d-1), \frac{fh}{t}\right],
\ee
which is derived from the first fundamental theorem for quasi-Riordan arrays (FFTQRA),

\be\label{1.12}
[g,f]u=gu(0)+\frac{f}{t}(u-u(0)).
\ee
Hence, $[1,t]$ is the identity of $q{\cR}$. 
\end{theorem}

In \cite{BHP}, it has been shown that the quasi-Riordan group is a normal subgroup of the almost-Riordan group. In Slowik and the author's paper \cite{HS23}, the total positivity of quasi-Riordan arrays is discussed. For instance, let $(g(t), f(t))$ be a Riordan array, where $g(t)=\sum_{n\geq 0} g_nt^n$ and $f(t)=\sum_{n\geq 1}f_nt^n$. If the lower triangular matrix 

\begin{equation}\label{eq:1}
\begin{array}{rl}
Q & =[g,f]=(g|f/t,t)=
\left [ \begin{array}{llllll} g_0& 0& 0& 0& 0& \cdots\\
g_1& f_1 & 0& 0& 0&\cdots\\
g_2 &f_2& f_1& 0& 0& \cdots\\
g_3& f_3& f_2& f_1&0&\cdots\\
\vdots &\vdots& \vdots& \vdots&\vdots&\ddots\end{array}\right]
\\
& =\left( g(t), f(t), tf(t), t^{2}f(t),\ldots\right)
\\
\end{array}
\end{equation}
is totally positive (TP), then so is $R=(g,f)$ (cf. \cite{MMW} or \cite{He, HS23}).  Other interesting criteria for total positivity of Riordan arrays can be found in \cite{CLW,CW}.

In Slowik and the author's another paper \cite{HS24}, a simdirect product is given and used to discuss the total positivity of almost-Riordan arrays via the total positivity of quasi-Riordan arrays.

\begin{theorem}\label{thm:2.5}\cite{HS24}
Every almost-Riordan array $(d|\,g,f)$ can be written as the semidirect product 

\be\label{eq:R_factoriz-1}
(d|\, g,f)=[d,tg](1|\,1,f),
\ee
or equivalently, 

\begin{equation}
\label{eq:R_factoriz}
(d|\,g,f)=\left[\begin{array}{c|c}
d_0 & 0  \\
\hline 
\frac{d-d_0}{t}&  (g,t)\\
\end{array}\right]
\left[\begin{array}{c|c}
1 & 0\\
\hline 
0 & (1,f)\\
\end{array}\right].
\end{equation}
\end{theorem}

The paper will be organized as follows. In next section, we continue the work shown in \cite{He25} to discuss double almost-Riordan arrays and the double almost-Riordan group in a different view. In Section $3$, we continue the work of \cite{He25} on the sequence characterizations of double almost-Riordan arrays and subgroups of the double almost-Riordan group and present production matrices of double almost-Riordan arrays by using a different approach. In Section $4$, we discuss the compression of double almost-Riordan arrays and present their sequence characterization. Finally, in Section $5$, we give a characteristic for the total positivity of double Riordan arrays, by using which we discuss the total positivity for several double almost-Riordan arrays. 

\section{Double almost-Riordan arrays}

\begin{definition}\label{def:3.1}
Let $b, g\in \F_0$ and $f_1,f_2\in \F_1$, where $g$ is even and $f_i$ are odd, i.e., $g=\sum_{n\geq 0} g_{2n} t^{2n}$ and $f_i=\sum_{n\geq 0} f_{i, 2n+1} t^{2n+1}$, 
$i=1,2$. Then we define a double almost-Riordan array by

\be\label{3.1}
(b|g;f_1,f_2)=(b, tg, tgf_1, tgf_1f_2, tg f_1^2f_2, tg f_1^2f_2^2,\ldots).
\ee
\end{definition}

\begin{theorem}\label{thm:3.2}
There two cases of the first fundamental theorems of double almost-Riordan arrays:

\be\label{3.2}
(b|g; f_{1}, f_{2})u(t)=v(t),
\ee
where for $u(t)=\sum_{k\geq 0}u_{2k}t^{2k}$ and $u(t)=\sum_{k\geq 0}u_{2k+1}t^{2k+1}$, we have 

\begin{align}
v(t)=&u_0b+\frac{tg}{f_2}(u(\sqrt{f_{1}f_{2}})-u_0)\quad \mbox{and}\label{3.3}\\
v(t)=&\frac{tg}{\sqrt{f_1f_2}}u(\sqrt{f_{1}f_{2}}),\label{3.4}
\end{align}
respectively. Here, \eqref{3.2} and \eqref{3.3} represents the first case of the first fundamental theorem of double almost-Riordan arrays (the first case of the FTDAR), and \eqref{3.2} and \eqref{3.4} represents the second case of the first fundamental theorem of double almost-Riordan arrays (the second case of the FTDAR). 
\end{theorem}

\begin{proof} 
If $u(t)=\sum_{k\geq 0}u_{2k}t^{2k}$, then by \eqref{3.2} 

\begin{align*}
&(b|g;f_1,f_2)u(t)\\
=&(b, tg, tgf_1, tgf_1f_2, tgf_1(f_1f_2), tg(f_1f_2)^2,\ldots) (u_0, 0, u_2, 0, u_4, 0, u_6,0\cdots)^T\\
=&u_0b+tgf_1(u_2+u_4(f_1f_2)+u_6(f_1f_2)^2+\cdots\\
=&u_0b +tgf_1\sum_{k\geq 1} u_{2k}(\sqrt{f_1f_2})^{2k-2}\\
=&u_0b+\frac{tg}{f_2}\sum_{k\geq 1}u_{2k}(\sqrt{f_1f_2})^{2k},
\end{align*}
in which the rightmost term implies \eqref{3.3}. Similarly, if $u(t)=\sum_{k\geq 0}u_{2k+1}t^{2k+1}$, 
then from \eqref{3.2} 

\begin{align*}
&(b|g;f_1,f_2)u(t)\\
=&(b, tg, tgf_1, tgf_1f_2, tgf_1(f_1f_2), tg(f_1f_2)^2,\ldots) (0, u_1, 0, u_3, 0, u_5, 0, \cdots)^T\\
=&tg(u_1+u_3(f_1f_2)+u_5(f_1f_2)^2+\cdots\\
=&tg\sum_{k\geq 0} u_{2k+1}(\sqrt{f_1f_2})^{2k}\\
=&\frac{tg}{\sqrt{f_1f_2}}\sum_{k\geq 0}u_{2k+1}(\sqrt{f_1f_2})^{2k+1},
\end{align*}
in which the rightmost term implies \eqref{3.4}. 
\end{proof}

By using the first fundamental theorems of double almost-Riordan arrays, we may establish a multiplication operator in the set of double almost-Riordan arrays. 

\begin{theorem}\label{thm:3.3}
Let $(b|g;f_1,f_2)$ and $(c|d;h_1,h_2)$ be two double almost-Riordan arrays, where $b$ and $c$ are even formal power series, $b=\sum_{k\geq 0} b_{2k} t^{2k}$ and $c=\sum_{k\geq 0} c_{2k}t^{2k}$. Then 

\begin{align}\label{3.6}
&(b|g;f_1,f_2)(c|d;h_1,h_2)\nonumber\\
=&( c_0b+\frac{tg}{f_2}(c(\sqrt{f_1f_2})-c_0)| gd(\sqrt{f_1f_2}); \sqrt{\frac{f_1}{f_2}}h_1(\sqrt{f_1f_2}),\sqrt{\frac{f_2}{f_1}}h_2(\sqrt{f_1f_2})),
\end{align}
and 

\begin{align}\label{3.7}
&(b|g;f_1,f_2)^{-1}\nonumber\\
=&\left( \frac{1}{b_0}+\frac{f_2(\overline{f_{12}})}{b_0\overline{f_{12}}g(\overline{f_{12}})}(b_0-b(\overline{f_{12}}))|
\frac{1}{g(\overline{f_{12}})};
\overline{f_{12}}\frac{t}{f_1(\overline{f_{12}})},\overline{f_{12}}\frac{t}{f_2(\overline{f_{12}})}
\right),
\end{align}
where $b_0\not= 0$ because $b\in \F_0$, and $\overline{f_{12}}$ is the compositional inverse of $\sqrt{f_1f_2}$, i.e., 

\[
\overline{f_{12}}(\sqrt{f_1f_2})=\sqrt{f_1f_2}(\overline{f_{12}})=t.
\] 
\end{theorem}
\begin{proof}

By using \eqref{3.6} with 

\[
c(t)= \frac{1}{b_0}+\frac{f_2(\overline{f_{12}})}{b_0\overline{f_{12}}g(\overline{f_{12}})}(b_0-b(\overline{f_{12}}))
\]
where $c_0=1/b_0$, we have 

\begin{align*}
&(b|g;f_1,f_2)\left( \frac{1}{b_0}+\frac{f_2(\overline{f_{12}})}{b_0\overline{f_{12}}g(\overline{f_{12}})}(b_0-b(\overline{f_{12}}))| \frac{1}{g(\overline{f_{12}})};
\overline{f_{12}}\frac{t}{f_1(\overline{f_{12}})},\overline{f_{12}}\frac{t}{f_2(\overline{f_{12}})}\right)\\
=&\left(\frac{b}{b_0}+\frac{tg}{f_2}\left( \frac{1}{b_0}+\frac{f_2}{b_0tg}(b_0-b)-\frac{1}{b_0}\right)| g\frac{1}{g}; t\sqrt{\frac{f_1}{f_2}}\frac{\sqrt{f_1f_2}}{f_1}, t\sqrt{\frac{f_2}{f_1}}\frac{\sqrt{f_1f_2}}{f_2}\right)\\
=&(1|1;t,t).
\end{align*}
To show that $(1|1;t,t)$ is the identity of the set of double almost-Riordan arrays, we calculate 

\[
(b|g;f_1,f_2)(1|1;t,t)=(b|g;f_1,f_2)=(1|1;t,t)(b|g;f_1,f_2)
\]
and complete the proof.
\end{proof}

From now on, for any $(b|g;f_1,f_2)\in {\cal D}a{\cal R}$, we always assume $b,g\in\F_0$ are even and $f_1,f_2\in\F_1$ are odd, namely, 

\begin{align*}
&b(t)=\sum_{k\geq 0} b_{2k}t^{2k},\quad g(t)=\sum_{k\geq 0} g_{2k}t^{2k},\quad \mbox{and}\\
&f_1(t)=\sum_{k\geq 0} f_{1,2k+1}t^{2k+1},\quad f_2(t)=\sum_{k\geq 0} f_{2,2k+1}t^{2k+1}.
\end{align*}

\begin{theorem}\label{thm:3.4}
The set of all double almost-Riordan arrays $(b|g;f_1f_2)$ with even $b,g\in \F_0$ and odd $f_1f_2\in \F_1$ forms a 
group called the double almost-Riordan group and denoted by ${\cD}a{\cR}$ under the multiplication operator defined by \eqref{3.6} with the identity $(1|1;t,t)$.
\end{theorem}

\begin{proof}
It is sufficient to show the following associativity for any elements $(b|g;f_1,f_2)$, $(c|d; h_1, h_2)$ and $(e|u; v_1, v_2)$:

\begin{align*}
&((b|g;f_1,f_2)(c|d; h_1, h_2))(e|u; v_1, v_2)\\
=&\left( c_0 b +\frac{tg}{f_2}(c(\sqrt{f_1f_2})-c_0)|gd(\sqrt{f_1f_2});
\sqrt{\frac{f_1}{f_2}}h_1(\sqrt{f_1f_2}), \sqrt{\frac{f_2}{f_1}}h_2(\sqrt{f_1f_2}
\right)(e|u;v_1v_2)\\
=& \left( e_0c_0b+\frac{tg}{f_2}\left( e_0(\sqrt{f_1f_2})+\frac{\sqrt{f_1f_2}d(\sqrt{f_1f_2})}{h_2(\sqrt{f_1f_2})}(e(\sqrt{h_1(\sqrt{f_1f_2})h_2(\sqrt{f_1f_2})})-e_0)-e_0c_0\right)| \right.\\
& tgd(\sqrt{f_1f_2})u(\sqrt{h_1(\sqrt{f_1f_2})h_2(\sqrt{f_1f_2})}); \sqrt{\frac{f_1h_1(\sqrt{f_1f_2})}{f_2h_2(\sqrt{f_1f_2})}}v_1(\sqrt{h_1(\sqrt{f_1f_2})h_2(\sqrt{f_1f_2})}), \\
&\left. \sqrt{\frac{f_2h_2(\sqrt{f_1f_2})}{f_1h_1(\sqrt{f_1f_2})}}v_2(\sqrt{h_1(\sqrt{f_1f_2})h_2(\sqrt{f_1f_2})})\right)\\
=& (b|g;f_1f_2)\left( e_0c+\frac{td}{h_2}(e(\sqrt{h_1h_2})-e_0)| du(\sqrt{h_1h_2}); \sqrt{\frac{h_1}{h_2}}v_1(\sqrt{h_1h_2}),\sqrt{\frac{h_1}{h_2}}v_1(\sqrt{h_1h_2})\right) \\
=& (b|g;f_1,f_2)((c|d; h_1, h_2)(e|u; v_1, v_2)).
\end{align*}
\end{proof}

The structures of the double almost-Riordan group is studied in \cite{He24}.

\section{Sequence characterization of double almost-Riordan arrays}

We now give a sequence characterization of a double almost Riordan array in ${\cD}a{\cR}$. Inspired by Branch, Davenport, Frankson, Jones, and Thorpe \cite{BDFJT} and  Davenport, Frankson, Shapiro, and Woodson \cite{DFSW}, we consider $D=(b|g;f_1,f_2)$ as 

\begin{align}\label{4.9}
D=&(b, tg, tgf, tg(f_1f_2), tgf_1(f_1f_2), tg(f_1f_2)^2, \ldots)\nonumber\\
=&(b,0,0,\ldots)+(0,tg, 0, tg(f_1f_2), 0, tg(f_1f_2)^2, 0, \ldots)\nonumber\\
&+(0,0,tgf_1, 0, tgf_1(f_1f_2), 0, tgf_1(f_1f_2)^2, 0\ldots)
\nonumber\\
=&D_0+D_1+D_2.
\end{align}
After omitting zero columns and top zero rows, we denote the remaining $D_1$ and $D_2$ shown above by $D_1^*$ and $D_2^*$, respectively. Then $D_1^*=(g, f_1f_2)=(d^{(1)}_{n,k})_{n,k\geq 0}$ and $D_2^*=(gf_1/t, f_1f_2)=(d^{(2)}_{n,k})_{n,k\geq 0}$ are Riordan arrays. Hence, $(b|g;f_1,f_2)$ has a $W$-sequence, two $Z$-sequences, denoted by $Z_1$- and $Z_2$-sequences, and a $A$-sequence in this new view, while the view shown in \cite{He24} provides $W$-sequence, a $Z$-sequnce, and two $A$-sequences: $A_1$-and $A_2$-sequences. 

By using \eqref{4.10} and \eqref{4.11}, we have the following result.

\begin{theorem}\label{thm:new-3.1}
Let $(b|g;f_1,f_2)$ be a double almost-Riordan array, and let $A(t)=\sum_{k\geq 0}a_kt^{2k}$, $Z_1(t)=\sum_{k\geq 0} z_{1,k} t^{2k}$, $Z_2(t)=\sum_{k\geq 0} z_{2,k} t^{2k}$ and $W(t)=\sum_{k\geq 0} w_k t^{2k}$ be the generating functions of $A$-, $Z_1$-, $Z_2$-, and $W$-sequences, respectively. Then

\begin{align}\label{8.0}
&A(t)=\frac{t^2}{\overline f_{12}^2},\\
&Z_1(t)=  \frac{1}{\overline{f_{12}}^2}\left( 1-\frac{g_0}{g(\overline{f_{12}})}\right)\label{8.1}\\
&Z_2(t)=  \frac{g_0f_{1,1}}{b_0}+\frac{t^2}{\overline{f_{12}}^2}-\frac{g_0f_{1,1}b(\overline{f_{12}})f_2(\overline{f_{12}})}{b_0\overline{f_{12}}g(\overline{f_{12}})},\label{8.2}\\
&W(t)=\frac{f_2(\overline{f_{12}})((1-w_0\overline{f_{12}}^2)b(\overline{f_{12}})-b_0)}{\overline{f_{12}}^3g(\overline{f_{12}})}+w_0,\quad w_0=b_2/b_0,\label{8.3}
\end{align}
where $f_{1,1}=[t] f_1(t)$, and $\overline{f_{12}}$ is the compositional inverse of $\sqrt{f_1f_2}$.
\end{theorem}

\begin{proof} By the view of $D^*_i$, $i=1,2$, shown above, for $k\geq 2$ we have 

\be\label{8.0-0}
d_{n,k}=a_0d_{n-2,k-2}+a_1d_{n-2, k}+a_2d_{n-2, k+2}+\cdots =\sum_{j\geq 0} a_j d_{n-2, k+2(j-1)},
\ee
or equivalently, for $k=2\ell+1$ 

\be\label{4.-1}
[t^n]tg(f_1f_2)^{\ell}=[t^{n-2}]tg \sum_{j\geq 0} a_j(f_1f_2)^{\ell+j-1}=[t^{n-2}]tg(f_1f_2)^{\ell-1}A(\sqrt{f_1f_2}),
\ee
and for $k=2\ell$

\be\label{4.-2}
[t^n]tgf_1(f_1f_2)^{\ell}=[t^{n-2}]tg f_1\sum_{j\geq 0} a_j(f_1f_2)^{\ell+j-1}=[t^{n-2}]tgf_1(f_1f_2)^{\ell-1}A(\sqrt{f_1f_2}).
\ee
Hence, from both \eqref{4.-1} and \eqref{4.-2}, we have 

\[
t^2A(\sqrt{f_1f_2})=f_1f_2,
\]
Substituting $t=\overline{f_{12}}$, the compositional inverse ot $\sqrt{f_1f_2}$ into the last equation and noting $(f_1f_2)(\overline {f_{12}})=t^2$ yields 

\[
\overline{f_{12}}^2A(t)=t^2,
\]
i.e., \eqref{8.0}. 

Similarly, from 

\[
d_{n,1}=z_{1,0}d_{n-2,1}+z_{1,1}d_{n-2, 3}+z_{1,2}d_{n-2, 5}+\cdots =\sum_{j\geq 0} z_{1,j} d_{n-2, 2j+1},
\]
we have 

\[
[t^n]tg=[t^{n-2}]tg\sum_{j\geq 0} z_{1,j}(\sqrt{f_1f_2})^{2j},
\]
which implies 

\[
\frac{t(g-g_0)}{t^2}=tgZ_1(\sqrt{f_1f_2}).
\]
Thus, 

\[
g(t)=\frac{g_{0}}{1-t^2Z_1(\sqrt{f_1f_2})},
\]
or equivalently, 

\[
Z_1(\sqrt{f_1f_2})= \frac{1}{t^2}\left( 1-\frac{g_0}{g(t)}\right),
\]
which implies \eqref{8.1}. 

By using 

\[
d_{n,2}=z_{2,0}d_{n-2,0}+z_{2,1}d_{n-2,2}+z_{2,2}d_{n-2,4}+\ldots, 
\]
where $z_{2,0}=f_{1,1}g_0/b_0$ with $f_{1,1}=f'_1(0)$, we have $Z_2(t)$, the generating function of the $Z_2$-sequence $(z_{2,0}, z_{2,1},z_{2,2},\ldots)$, for $D^*_2$ satisfying

\begin{align*}
[t^n]tgf_1=&[t^{n-2}]\left( z_{2,0}b+tgf_1(z_{2,1}+z_{2,2}(f_1f_2)+z_{2,3}(f_1f_2)^2+\cdots)\right)\\
=&[t^{n}]t^2\left(\frac{f_{1,1}g_0}{b_0}b+\frac{tg}{f_2}\left(Z(\sqrt{f_1f_2})-\frac{f_{1,1}g_0}{b_0}\right)\right).
\end{align*}
Thus,

\begin{align*}
Z_2(\sqrt{f_1f_2})=& \frac{f_{1,1}g_0}{b_0}+\frac{f_2}{t^2g} \left(gf_1-tb\frac{f_{1,1}g_0}{b_0}\right)\\
=&\frac{f_{1,1}g_0}{b_0}+\frac{1}{b_0t^2g}\left(b_0gf_1f_2-f_{1,1}g_0tbf_2\right),
\end{align*}
which implies \eqref{8.2}. 

Let $(d_{n,k})_{n,k\geq 0}$ be any infinite lower triangular array with $d_{n,n}\not= 0$, and let its first column has even generating function $b(t)=\sum_{k\geq 0} b_kt^{2k}$. Then we denote $w_0=d_{2,0}/d_{0,0}=b_2/b_0$ and use 

\be\label{8.4}
d_{n,0}=w_{0}d_{n-2,0}+w_{1}d_{n-2,2}+w_{2}d_{n-2,4}+\cdots 
\ee
to determine $w_1$, $w_2, \ldots$, uniquely and recursively. For instance, 

\[
w_1=\frac{d_{4,0}-w_0d_{2,0}}{d_{2,2}}=\frac{d_{0,0}d_{4,0}-d_{2,0}^2}{d_{0,0}d_{2,2}},
\]
etc. Finally, the equivalence of \eqref{8.3} and \eqref{8.4} can be seen from the equivalence of \eqref{8.4} and 

\[
[t^n]b=[t^n]t^2\left(w_0b+tgf_1(w_1+w_2(f_1f_2)+\cdots\right),\quad n\geq 2.
\]
The last expression is 

\[
\frac{b(t)-b_0}{t^2}=w_0b+\frac{tg}{f_2}(W(\sqrt{f_1f_2})-w_0),
\]
or equivalently, equation \eqref{8.3}. 
\end{proof}

We may using production matrix to represent the sequences characterizations of $(b|g;f_1,f_2)$ in terms of the view shown in \cite{BDFJT, DFSW}. 

\begin{theorem}\label{thm:4.3-2}
Let $(b|g;f_1,f_2)\in {\cD}a{\cR}$. Then, $(b|g;f_1,f_2)$ has a production matrix 

\be\label{4.2-2}
P=\left( W(t), tZ_1(t), Z_2(t), tA(t), t^2A(t), t^3A(t), \ldots \right),
\ee
where

\begin{align}
& W(t)=\frac{b_2}{b_0}+\frac{f_2(\overline{f_{12}})}{b_0\overline{f_{12}}^3g(\overline{f_{12}})}(b(\overline{f_{12}})(b_0-b_2\overline{f_{12}}^2)-b_0^2),\label{4.10}\\
&Z_1(t)=\frac{1}{\overline{f_{12}}^2}\left(1-\frac{g_0}{g(\overline{f_{12}})}\right),\label{4.11}\\
&Z_2(t)=\frac{g_0f_{1,1}}{b_0}+\frac{t^2}{\overline{f_{12}}^2}-\frac{g_0f_{1,1}b(\overline{f_{12}})f_2(\overline{f_{12}})}{b_0\overline{f_{12}}g(\overline{f_{12}})},\label{4.12}\\
%\frac{1}{b_0\overline{f_{12}}^2g(\overline{f_{12}})}\left(b_0t^2g(\overline{f_{12}})-f_{1,1}g_0\overline{f_{12}}b(\overline{f_{12}})f_2(\overline{f_{12}})\right),\label{4.12}\\
&A(t)=\frac{t^2}{\overline{f_{12}}^2},\label{4.13}
\end{align}
where $f_{1,1}=[t] f_1(t)$, and $\overline{f_{12}}$ is the compositional inverse of $\sqrt{f_1f_2}$.
\end{theorem}

\begin{proof}
Since 

\[
\overline{\overline{(b|g;f_1,f_2)}}=\left( \frac{b-b_0}{t^2}, \frac{g-g_0}{t}, \frac{gf_1}{t}, \frac{gf_1f_2}{t}, \frac{gf_1(f_1f_2)}{t}, \frac{g(f_1f_2)^2}{t},\ldots\right),
\]
where $\overline{\overline{(b|g;f_1,f_2)}}$ is the truncated $(b|g;f_1,f_2)$ with the first two rows omitted, the production matrix shown in \eqref{4.2-2} is equal to 

\begin{align*}
P=&(b|g;f_1,f_2)^{-1}\overline{\overline{(b|g;f_1,f_2)}}\\
=&\left(c|\frac{1}{g(\overline{f_{12}})};\overline{f_{12}}\frac{t}{f_1(\overline{f_{12}})},\overline{f_{12}}\frac{t}{f_2(\overline{f_{12}})}\right)\\
&\left( \frac{b-b_0}{t^2}, \frac{g-g_0}{t}, \frac{gf_1}{t}, \frac{gf_1f_2}{t}, \frac{gf_1(f_1f_2)}{t}, \frac{g(f_1f_2)^2}{t},\ldots\right),
\end{align*}
where $c$ is given in \eqref{3.7}, i.e., 

\[
c=\frac{1}{b_0}+\frac{f_2(\overline{f_{12}})}{b_0\overline{f_{12}}g(\overline{f_{12}})}(b_0-b(\overline{f_{12}})).
\] 
Noting the generating function $(b-b_0)/t^2$ of the column zero of $\overline{\overline{(b|g;f_1,f_2)}}$ is even, 
we may use the first case of the FFTDAR represented in Theorem \ref{thm:3.2} and the expression of function $c$ to obtain 

\begin{align*}
(b|g;f_1,f_2)^{-1}\frac{b-b_0}{t^2}=&b_2c+\frac{t\frac{1}{g(\overline{f_{12}})}}{\overline{f_{12}}\frac{t}{f_2(\overline{f_{12}})}}\left(\frac{b(\overline{f_{12}})-b_0}{\overline{f_{12}}^2}-b_2\right)\\
=&b_2\left(\frac{1}{b_0}+\frac{f_2(\overline{f_{12}})}{b_0\overline{f_{12}}g(\overline{f_{12}})}(b_0-b(\overline{f_{12}}))\right)\\
&+\frac{f_2(\overline{f_{12}})}{\overline{f_{12}}g(\overline{f_{12}})}\left(\frac{b(\overline{f_{12}})-b_0}{\overline{f_{12}}^2}-b_2\right)\\
=&\frac{b_2}{b_0}+\frac{f_2(\overline{f_{12}})}{b_0\overline{f_{12}}^3g(\overline{f_{12}})}(b(\overline{f_{12}})(b_0-b_2\overline{f_{12}}^2)-b_0^2),
\end{align*}
which gives the generating function of the first column (column zero) of $P$ and the generating function of the $W$-sequence, as shown in \eqref{4.10}.

Noting the generating function $(g-g_0)/t$ of the column one of $\overline{\overline{(b|g;f_1,f_2)}}$ is odd, we may use the second case of the FFTDAR shown in Theorem \ref{thm:3.2} to find 

\begin{align*}
(b|g;f_1,f_2)^{-1}\frac{g-g_0}{t}=&\left(c|\frac{1}{g(\overline{f_{12}})};\overline{f_{12}}\frac{t}{f_1(\overline{f_{12}})},\overline{f_{12}}\frac{t}{f_2(\overline{f_{12}})}\right)\frac{g-g_0}{t}\\
=&\frac{t\frac{1}{g(\overline{f_{12}})}}{\overline{f_{12}}}\frac{g(\overline{f_{12}})-g_0}{\overline{f_{12}}}=\frac{t}{\overline{f_{12}}^2}\left( 1-\frac{g_0}{g(\overline{f_{12}})}\right),
\end{align*}
which gives the generating function of the second column (column one) of $P$ and the generating function of the $Z_1$-sequence, as shown in \eqref{4.11}.

Similarly, since $gf_1/t$ is even, by using the first case of the FFTDAR, we get 

\begin{align*}
(b|g;f_1,f_2)^{-1}\frac{gf_1}{t}=&\left(c|\frac{1}{g(\overline{f_{12}})};\overline{f_{12}}\frac{t}{f_1(\overline{f_{12}})},\overline{f_{12}}\frac{t}{f_2(\overline{f_{12}})}\right)\frac{gf_1}{t}\\
=&g_0f_{1,1}c+\frac{t\frac{1}{g(\overline{f_{12}})}}{\overline{f_{12}}\frac{t}{f_2(\overline{f_{12}})}}\left(\frac{g(\overline{f_{12}})f_1(\overline{f_{12}})}{\overline{f_{12}}}-g_0f_{1,1}\right)\\
=&g_0f_{1,1}\left(\frac{1}{b_0}+\frac{f_2(\overline{f_{12}})}{b_0\overline{f_{12}}g(\overline{f_{12}})}(b_0-b(\overline{f_{12}}))\right)\\
& +\frac{f_2(\overline{f_{12}})}{\overline{f_{12}}g(\overline{f_{12}})}\left(\frac{g(\overline{f_{12}})f_1(\overline{f_{12}})}{\overline{f_{12}}}-g_0f_{1,1}\right)\\
=&\frac{g_0f_{1,1}}{b_0}-\frac{g_0f_{1,1}b(\overline{f_{12}})f_2(\overline{f_{12}})}{b_0\overline{f_{12}}g(\overline{f_{12}})}
+\frac{f_1(\overline{f_{12}})f_2(\overline{f_{12}})}{\overline{f_{12}}^2},
\end{align*}
which gives the generating function of the third column (column two) of $P$ and the generating function of the $Z_2$-sequence, as shown in \eqref{4.12}.

Since $gf_1f_2/t$ is odd, we use the first case of the FFTDAR to calculate 

\begin{align*}
(b|g;f_1,f_2)^{-1}\frac{gf_1f_2}{t}=&\left(c|\frac{1}{g(\overline{f_{12}})};\overline{f_{12}}\frac{t}{f_1(\overline{f_{12}})},\overline{f_{12}}\frac{t}{f_2(\overline{f_{12}})}\right)\frac{gf_1f_2}{t}\\
=&\frac{t\frac{1}{g(\overline{f_{12}})}}{\overline{f_{12}}}\left( \frac{g(\overline{f_{12}})f_1(\overline{f_{12}})f_2(\overline{f_{12}})}{\overline{f_{12}}}\right)=\frac{t^3}{\overline{f_{12}}^2},
\end{align*}
which gives the generating function of the fourth column (column three) of $P$ and $t$ multiples of the generating function of the $A$-sequence as shown in \eqref{4.13}.

In general, we may find the generating function of the column $k-1$, $k\geq 4$, of $P$ is  

\[
\frac{t^{k}}{\overline{f_{12}}^2}=t^{k-2}A(t),
\]
which completes the proof of the theorem.
\end{proof}

An alternative sequence characterization and the corresponding production matrix for double almost-Riordan arrays based on Theorem \ref{thm:0.2} are given in \cite{He24}.

\begin{example}\label{ex:2.1} $\hat R$ shown in \eqref{1.5-4} is the compression of the double Riordan array, the Fibonacci-Stanley array, $(1/(1-t^4)|1/(1-t^2); t, t/(1-t^2))$, in which $(1/(1-t^2);t, t/(1-t^2))$ presents the Fibonacci-Stanley tree. 
Here, $f_1=t$, $f_2=t/(1-t^2)$, $g=1/(1-t^2)$, and $b=1/(1-t^4)$.  Thus, $\sqrt{f_1f_2}=t/\sqrt{1-t^2}$, and the compositional inverse of $\sqrt{f_1f_2}$ is $\overline{f_{12}}=t/\sqrt{1+t^2}$. Substituting $f_1=t$ $f_2=t/(1-t^2)$, $g=1/(1-t^2)$, and $b=1/(1-t^4)$ into Equations \eqref{4.10}-\eqref{4.13} and noting 

\begin{align*}
&f_1(\overline{f_{12}})=\frac{t}{\sqrt{1+t^2}},\\
&f_2(\overline{f_{12}})=t\sqrt{1+t^2},\\
&g(\overline{f_{12}})=1+t^2,\\
&b(\overline{f_{12}})=\frac{(1+t^2)^2}{1+2t^2},
\end{align*}
we immediately have 

\begin{align*}
&A(t)=\frac{t^2}{\overline{f_{12}}^2}=1+t^2,\\
&Z_1(t)=\frac{1}{(t/\sqrt{1+t^2})^2}\left(1-\frac{1}{1+t^2}\right)=1,\\
&Z_2(t)=\frac{t\sqrt{1+t^2}\left( 1+t^2-\frac{(1+t^2)^2}{1+2t^2}\right)}{\frac{t}{\sqrt{1+t^2}}(1+t^2)}+1=\frac{1+3t^2+t^4}{1+2t^2},\\
&\qquad = 1+t^2-t^4+2t^6-4t^8+8t^{10}+\cdots,\\
&W(t)=\frac{t\sqrt{1+t^2}\left(\frac{(1+t^2)^2}{1+2t^2}-1\right)}{\left(\frac{t}{\sqrt{1+t^2}}\right)^3(1+t^2)}+0=\frac{t^2(1+t^2)}{1+2t^2}\\
&\qquad =t^2-t^4+2t^6-4t^8+8t^{10}+\cdots.
\end{align*}
The double almost-Riordan array $(1/(1-t^4)|1/(1-t^2); t, t/(1-t^2))$ begins 

\be
%\begin{bmatrix}
\left[\begin{array}{lllllllllll}
1&0&0&0&0&0&0&0&0&0&...\\
0&1&0&0&0&0&0&0&0&0&...\\
0&0&1&0&0&0&0&0&0&0&...\\
0&1&0&1&0&0&0&0&0&0&...\\
1&0&1&0&1&0&0&0&0&0&...\\
0&1&0&2&0&1&0&0&0&0&...\\
0&0&1&0&2&0&1&0&0&0&...\\
0&1&0&3&0&3&0&1&0&0&..\\
1&0&1&0&3&0&3&0&1&0&...\\
0&1&0&4&0&6&0&4&0&1&...\\
\vdots&\vdots&\vdots&\vdots&\vdots&\vdots&\vdots&\vdots&\vdots&\vdots&\ddots
%\end{bmatrix} 
\end{array}\right]
\ee

Denote 

\begin{align*}
P=&(W(t), tZ_1(t), Z_2(t), tA(t), t^2A(t), \ldots)\\
=&\left[\begin{array}{llllllllllll}
0& 0& 1&0&0&0&0&0&0&0&0&...\\
0& 1& 0&1&0&0&0&0&0&0&0&...\\
1& 0& 1&0&1&0&0&0&0&0&0&...\\
0& 0 &0&1&0&1&0&0&0&0&...\\
-1&0&-1&0&1&0&1&0&0&0&0&...\\
0& 0&0&0&0&1&0&1&0&0&0&...\\
2& 0& 2&0&0&0&1&0&1&0&0&...\\
0&0&  0&0&0&0&0&1&0&1&0&..\\
-4&0&-4&0&0&0&0&0&1&0&1&...\\
0& 0& 0&0&0&0&0&0&0&1&0&...\\
\vdots&\vdots&\vdots&\vdots&\vdots&\vdots&\vdots&\vdots&\vdots&\vdots&\ddots
\end{array}\right].
\end{align*}

We find 

\begin{align*}
(b|g;f_1.f_2)P=
&\overline{\overline{\left( \frac{1}{1-t^4}|\frac{1}{1-t^2}; t,\frac{t}{1-t^2}\right)}},
\end{align*}
where the rightmost matrix is the truncation of 

\[
(b|g;f_1,f_2)=\left( \frac{1}{1-t^4}|\frac{1}{1-t^2}; t, \frac{t}{1-t^2}\right)
\]
with the first row and the second element of column zero omitted. 
\end{example}

\section{Compressions of double almost-Riordan arrays} 

Let $(b|g;f_1,f_2)=(d_{n,k})_{n\geq k\geq 0}\in {\cal D}a{\cal R}$. We define its compression $(\hat d_{n,k})_{n\geq k\geq 0}$ as follows:

\be\label{9.1} 
\hat d_{n,k}:=d_{2n-k,k}, \quad n\geq k\geq 0. 
\ee
We now study the structure of the compression of a Riordan array starting from the following theorem. 

\begin{theorem}\label{thm:9.2}
Let $(b|g;f_1,f_2)=(d_{n,k})_{n\geq k\geq 0}$ be a double almost-Riordan array with 

\begin{align*}
& b(t)=\sum_{k\geq 0} b_{2k}t^{2k}, \,\, g(t)=\sum_{k\geq 0} g_{2k}t^{2k},\,\,  f_1(t)=\sum_{k\geq 0}f_{1,2k+1}t^{2k+1}, \,\, \mbox{and} \nonumber\\ 
&  f_2(t)=\sum_{k\geq 0} f_{2,2k+1}t^{2k+1},
\end{align*}
and let its compression array $(\hat d_{n,k})_{n,k\geq 0}$ be defined by \eqref{9.1}. Then we have  

\begin{align}\label{9.1-2} 
&\hat d_{n,0}=[t^n] \hat b(t),\nonumber \\
&\hat d_{n,k}=\begin{cases} [t^n]t\hat g (\hat f_1\hat f_2)^{(k-1)/2}, & \mbox{if $k$ is odd},\\
[t^n]t\hat g \hat f_1(\hat f_1\hat f_2)^{(k-2)/2}, &\mbox{if $k$ is even}
\end{cases}
\end{align}
for $k\geq 1$, where 
\begin{align}\label{9.1-3}
&\hat b(t)=\sum_{k\geq 0} b_{2k}t^k, \,\, \hat g(t)=\sum_{k\geq 0} g_{2k}t^k,\,\, \hat f_1(t)=\sum_{k\geq 0}f_{1,2k+1}t^{k+1}, \,\, \mbox{and} \nonumber\\ 
& \hat f_2(t)=\sum_{k\geq 0} f_{2,2k+1}t^{k+1}.
\end{align}
\end{theorem}

\begin{proof} From \eqref{9.1}, 

\[
\hat d_{n,0}=d_{2n,0}=[t^{2n}]b(t)=[t^{2n}]\sum_{k\geq 0} b_{2k} t^{2k}=[t^n] \sum_{k\geq 0} b_{2k} t^k=[t^n] \hat b(t).
\]
For $k\geq 1$, we consider two cases of that $k$ is odd or $k$ is even, respectively. If $k=2\ell+1$, using \eqref{9.1} yields. 

\begin{align*}
\hat d_{n,2\ell+1}=& d_{2n-2\ell-1,2\ell +1}=[t^{2n-2\ell-1}]tg(f_1f_2)^{\ell}\\
=&[t^{2n}]t^{2\ell+2}g(f_1f_2)^\ell=[t^{2n}]t^2g(tf_1tf_2)^\ell.
\end{align*}
On the rightmost side, using the transformation $t^2\to t$ and noting $g(t^{1/2})=\hat g$ and $t^{1/2}f_i(t^{1/2})=\hat f_i$ 
($i=1,2$), we obtain 

\[
\hat d_{2n,2\ell+1}=[t^n]t\hat g(t)(\hat f_1\hat f_2)^\ell.
\]
Similarly, if $k=2\ell$, equation \eqref{9.1} shows 

\begin{align*}
\hat d_{n,2\ell }=&d_{2n-2\ell,2\ell}=[t^{2n-2\ell}]tgf_1(f_1f_2)^{\ell-1}=[t^{2n}]t^{2\ell+1}g(f_1f_2)^{\ell-1}\\
=&[t^{2n}]t^2g tf_1(tf_1 tf_2)^{\ell-1}.
\end{align*}
On the rightmost side, using the transformation $t^2\to t$ and noting $g(t^{1/2})=\hat g$ and $t^{1/2}f_i(t^{1/2})=\hat f_i$ 
($i=1,2$), we may write the above equation to 

\[
\hat d_{n,2\ell}=[t^n]t\hat g \hat f_1(\hat f_1\hat f_2)^{\ell-1},
\]
completing the proof.
\end{proof}

As an example, $\hat R=(1/(1-t), t, t/(1-t))$ shown in \eqref{1.5-4} is the compression of the double Riordan array $(1/(1-t^2), t, t/(1-t^2))$, and $(1/(1-t^2)|1/(1-t), t, t/(1-t))$ is the compression of the double almost-Riordan array shown in Example \ref{ex:2.1}.

The sequence characterization of the compression of a double Riordan array is given in the following theorem. 

\begin{theorem}\label{thm:9.3}
Let $(b|g;f_1,f_2)=(d_{n,k})_{n\geq k\geq 0}$ be a double Riordan array, and let its compression 
be defined by $(\hat d_{n,k})_{n\geq k\geq 0}$, where $\hat d_{n,k}$ is shown in \eqref{9.1}. 
Suppose the $A$-, $Z_1$-, $Z_2$-, and $W$-sequences of $(b|g;f_1,f_2)$ are 

\begin{align*}
&A=\{ a_0, a_1,\ldots\}, \,\,Z_1=\{ z_{1,0}, z_{1,1}, \ldots\}, \\
&Z_2=\{ z_{2,0}, z_{2,1}, \ldots\}, \,\, 
 \mbox{and}\,\, W=\{w_0,w_1,\ldots\}
\end{align*}
with their generating functions $A(t)=\sum_{k\geq 0}a_kt^{2k}$, $Z_1(t)=\sum_{k\geq 0} z_{1,k} t^{2k}$, $Z_2(t)=\sum_{k\geq 0} z_{2,k} t^{2k}$ and $W(t)=\sum_{k\geq 0} w_k t^{2k}$, respectively. Then for $k=1,2,\ldots$ and $n=1,2,\ldots$, 

\begin{align}\label{9.2}
&\hat d_{n,k}=a_{0}\hat d_{n-2,k-2}+a_{1}\hat d_{n-1,k}+a_{2}\hat d_{n,k+2}+\cdots=\sum_{j\geq 0} 
a_j\hat d_{n+j-2, k+2(j-1)},\quad k\geq 3\\
&\hat d_{n,2}=z_{2,0}\hat d_{n-2,0}+z_{2,1}\hat d_{n-1,2}+z_{2,2}\hat d_{n,4}+\cdots=\sum_{j\geq 0} z_{2,j}\hat d_{n+j-2, 2j},
\label{9.3}\\
&\hat d_{n,1}=z_{1,0}\hat d_{n-1,1}+z_{1,1}\hat d_{n,3}+z_{1,2}\hat d_{n+1,5}+\cdots=\sum_{j\geq 0} z_{1,j}\hat d_{n+j-1,2j+1},\label{9.4}\\
&\hat d_{n,0}=w_0\hat d_{n-1,0}+w_1\hat d_{n,2}+w_2\hat d_{n+1,4}+\cdots=\sum_{j\geq 0} w_j\hat d_{n+j-1, 2j}, \label{9.5},
\end{align}

or equivalently,

\begin{align}\label{9.2-2}
&\hat f_1\hat f_2=t^2A\left( \sqrt{\frac{\hat f_1\hat f_2}{t}}\right),\\
&Z_2\left(\sqrt{ \frac{\hat f_1\hat f_2}{t}}\right)=\frac{\hat f_2}{t^2\hat g}(\hat g\hat f_1-z_{2,0}t\hat b)+z_{2,0},\quad z_{2,0}=\frac{g_0f_{11}}{b_0},\label{9.3-2}\\
&Z_1\left(\sqrt{\frac{\hat f_1\hat f_2}{t}}\right)=\frac{1}{t}\left(1-\frac{g_0}{\hat g}\right), \label{9.4-2}\\
&W\left(\sqrt{\frac{\hat f_1\hat f_2}{t}}\right)=\frac{\hat f_2}{t^2\hat g}\left(\hat b(1-w_0t)-b_0\right)+w_0. \quad w_0=\frac{b_2}{b_0},\label{9.5-2}
\end{align}
\end{theorem}

\begin{proof}
To find the sequence characterization of the compression of a double Riordan array, we consider the cases of $\hat d_{n,2l-1}$ and $\hat d_{n,2l}$ for $l=1,2,\ldots$ as well as $\hat d_{n,1}$ and $\hat d_{n,0}$ for $n\geq 1$. From \eqref{9.1} and \eqref{8.0-0}, we have 

\begin{align}\label{9.0}
\hat d_{n,k}=&d_{2n-k,k}=\sum_{j\geq 0} a_jd_{2(n-1)-k,k+2(j-1)}\nonumber\\
=&\sum_{j\geq 0}a_j d_{2(n+j-2)-k-2(j-1), k+2(j-1)}\nonumber\\
=&\sum_{j\geq 0} a_j\hat d_{n+j-2, k+2(j-1)}.
\end{align}
Thus, for $k=2\ell-1$, $\ell\geq 2$, we have 

\[
[t^n]t\hat g(\hat f_1\hat f_2)^{\ell-1}=\hat d_{n,2\ell-1}=\sum_{j\geq 0} a_j\hat d_{n+j-2, 2(\ell+j-2)+1}=\sum_{j\geq 0} a_j [t^{n+j-2}]t\hat g(\hat f_1\hat f_2)^{\ell+j-2},
\]
which proves \eqref{9.2} for $k=2\ell-1$, $\ell\geq 2$, and 

\[
t\hat g(\hat f_1\hat f_2)^{\ell-1}=t\hat g(\hat f_1\hat f_2)^{\ell-2}t^2\sum_{j\geq 0} a_j \left(\frac{\hat f_1\hat f_2}{t}\right)^j,
\]
Hence, we obtain \eqref{9.2-2} for $k=2\ell -1$, $\ell\geq 2$.

Similarly, for $k=2\ell$, $\ell\geq 2$, we have 

\[
[t^n]t\hat g\hat f_1(\hat f_1\hat f_2)^{\ell-1}=\hat d_{n,2\ell}=\sum_{j\geq 0} a_j\hat d_{n+j-2, 2(\ell+j-1)}=\sum_{j\geq 0} a_j [t^{n+j-2}]t\hat g\hat f_1(\hat f_1\hat f_2)^{\ell+j-2},
\]
which gives \eqref{9.2} for $k=2\ell$, $\ell\geq 2$, and 

\[
t\hat g\hat f_1(\hat f_1\hat f_2)^{\ell-1}=t\hat g\hat f_1(\hat f_1\hat f_2)^{\ell-2}t^2\sum_{j\geq 0} a_j \left(\frac{\hat f_1\hat f_2}{t}\right)^j,
\]
Hence, we obtain \eqref{9.2-2} for $k=2\ell$, $\ell\geq 2$. We have proved \eqref{9.2} and \eqref{9.2-2} for all 
$k\geq 1$. 

Considering the entries in column $2$, we have 

\begin{align*}
\hat d_{n,2}=&d_{2n-2,2}=z_{2,0}d_{2n-4, 0}+z_{2,1}d_{2n-4,2}+z_{2,2}d_{2n-4,4}+\ldots\\
=& z_{2,0}\hat d_{n-2, 0}+z_{2,1}\hat d_{n-1,2}+z_{2,2}\hat d_{n,4}+\ldots,
\end{align*}
namely, \eqref{9.3}. Thus 

\[
[t^n]t\hat g\hat f_1=[t^{n-2}]z_{2,0} b+[t^{n-2+j}]t\hat g\hat f_1\sum_{j\geq 1} z_{2,j}(\hat f_1\hat f_2)^{j-1},
\]
or equivalently,

\[
t\hat g\hat f_1=z_{2,0}t^2\hat b+\frac{t^3\hat g}{f_2}\left( Z_2\left(\sqrt{\frac{\hat f_1\hat f_2}{t}}\right) -z_{2,0}\right),
\]
which gives \eqref{9.3-2}.

Similarly, for the entries of column $1$ 

\begin{align*} 
\hat d_{n,1}=&d_{2n-1,1}=z_{1,0} d_{2n-3,1}+z_{1,1}d_{2n-3,3}+z_{1,2}d_{2n-3,5}+\ldots\\
=&z_{1,0}\hat d_{n-1,1}+z_{1,1}\hat d_{n, 3}+z_{1,2}\hat d_{n+1,5}+\ldots,
\end{align*}
which gives \eqref{9.4} and implies 

\[
[t^n]t\hat g=[t^{n-1+j}]t \hat g\sum_{j\geq 0}z_{1,j}(\hat f_1\hat f_2)^j),
\]
or equivalently,

\[
t(\hat g-g_0)=t^2\hat gZ_1\left( \sqrt{\frac {\hat f_1\hat f_2}{t}}\right).
\]
From the last equation, we find \eqref{9.4-2}.

Finally, we consider entries of column $0$ as

\begin{align*}
\hat d_{n,0}=&d_{2n,0}=w_0d_{2n-2,0}+w_1d_{2n-2,2}+w_2d_{2n-2,4}+\ldots\\
=& w_0\hat d_{n-1}+w_1\hat d_{n,2}+w_2\hat d_{n+1, 4}+\ldots,
\end{align*}
which gives \eqref{9.5} and 

\[
[t^n]\hat b=[t^{n-1}]w_0\hat b+[t^{n-1+j}]\sum_{j\geq 1}w_{j}t\hat g \hat f_1(\hat f_1\hat f_2)^{j-1}.
\]
Hence, 

\[
\hat b-b_0=w_0t \hat b+\frac{t^2\hat g}{\hat f_2}\left( W\left(\sqrt{\frac{\hat f_1\hat f_2}{t}}\right)-w_0\right),
\]
where $w_0=b_2/b_0$, which implies \eqref{9.5-2}. 
\end{proof}

\begin{remark}\label{rem:4.1}
It can be seen that the compositional inverse of $\sqrt{(\hat f_1\hat f_2)/t}$ is

\[
\overline {\sqrt{\frac{\hat f_1 \hat f_2}{t}}}=\overline{f_{1,2}}^2,
\]
where $\overline{f_{12}}$ is the compositional inverse of $\sqrt{f_1f_2}$. Hence, \eqref{9.2-2}-\eqref{9.5-2} are equivalently, respectively, \eqref{4.10}-\eqref{4.13}. For instance, substituting $t=\overline{f_{12}}^2$ into the expressions involving $A$ and $Z_1$, we obtain 

\begin{align*}
&A(t)=\frac{t^2}{\overline{f_{12}}^2},\\
&Z_1(t)=\frac{1}{\overline{f_{12}}^2}\left( 1-\frac{g_0}{g(\overline{f_{12}})}\right),
\end{align*}
respectively. To transform $Z_2$ in \eqref{9.3-2}, we use $t=\overline{f_{12}}^2$ and 

\be\label{0-0}
\overline {f_{12}}f_i(\overline {f_{12}})=\hat f_i(\overline{f_{12}}^2),\quad i=1,2,
\ee
into \eqref{9.3-2} and notice $f_1(\overline{f_{12}})f_2(\overline{f_{12}})=t^2$ to obtain

\begin{align*}
Z_2(t)=&\frac{\overline{f_{12}}f_2(\overline{f_{12}})}{\overline{f_{12}}^4g(\overline{f_{12}})}\left(\overline{f_{12}}f_1(\overline{f_{12}})g(\overline{f_{12}})-\frac{g_0f_{11}}{b_0}\overline{f_{12}}^2b(\overline{f_{12}})\right)+\frac{g_0f_{11}}{b_0}\\
=&\frac{g_0f_{11}}{b_0}+\frac{t^2}{\overline {f_{12}}^2}-\frac{g_0f_{11}b(\overline{f_{12}})f_2(\overline{f_{12}})}{b_0\overline{f_{12}}g(\overline{f_{12}})}.
\end{align*}
Hence, we get \eqref{4.11} from \eqref{9.3-2}.

Finally, by applying \eqref{0-0} and substituting $t=\overline{f_{12}}^2$ into \eqref{9.4-2}, we have 

\[
W(t)=\frac{\overline{f_{12}}f_2(\overline{f_{12}})}{\overline{f_{12}}^4g(\overline{f_{12}})}((1-w_0\overline{f_{12}}^2)b(\overline{f_{12}})-b_0)+w_0,\quad w_0=\frac{b_2}{b_0}, 
\]
which implies \eqref{4.13}.
\end{remark}

Equation  \eqref{9.2} can be presented 

\bns
&&[t^{n}]t\hat g \hat f_{1}(\hat f_{1}\hat f_{2})^{k -1}=\sum_{i\geq 0}a_{1,i}[t^{n-i-1}]t\hat g(\hat f_{1}\hat f_{2})^{k-1+i}\\
&=&[t^{n}]t\hat g(\hat f_{1}\hat f_{2})^{k -1}\sum_{i\geq 0}a_{1,i}t^{i+1}(\hat f_{1}\hat f_{2})^{i}.
\ens
The above expression has the following equivalent form with the using of the generating function of the columns of the compression matrix $(\hat d_{n,k})_{n,k\geq 0}$:

\[
\hat f_{1}=tA\left( t \hat f_{1}\hat f_{2}\right).
\]

In Example \ref{ex:2.1}, $\hat R$ shown in \eqref{1.5-4} is the compression of the double Riordan array, the Fibonacci-Stanley array, $(1/(1-t^4)|1/(1-t^2); t, t/(1-t^2))$, in which $(1/(1-t^2);t, t/(1-t^2))$ presents the Fibonacci-Stanley tree, where $(1/(1-t^2);t, t/(1-t^2))$ begins 

\[
\hat R=(\hat d_{n,k})_{n,k\geq 0}= 
\begin{bmatrix}
1&0&0&0&0&0&0&0&0&...\\
0&1&0&0&0&0&0&0&0&...\\
1&1&1&0&0&0&0&0&0&...\\
0&1&1&1&0&0&0&0&0&...\\
1&1&1&2&1&0&0&0&0&...\\
0&1&1&3&2&1&0&0&0&...\\
1&1&1&4&3&3&1&0&0&...\\
 \vdots&\vdots&\vdots&\vdots&\vdots&\vdots&\vdots&\vdots&\vdots&\ddots\\
\end{bmatrix} 
\]
Since the $A$-, $Z_1$-, $Z_2$-, and $W$- sequences of $(1/(1-t^4)|1/(1-t^2); t, t/(1-t^2))$ are 
\begin{align*}
&A(t)=\frac{t^2}{\overline{f_{12}}^2}=1+t^2,\\
&Z_1(t)=\frac{1}{(t/\sqrt{1+t^2})^2}\left(1-\frac{1}{1+t^2}\right)=1,\\
&Z_2(t)=\frac{t\sqrt{1+t^2}\left( 1+t^2-\frac{(1+t^2)^2}{1+2t^2}\right)}{\frac{t}{\sqrt{1+t^2}}(1+t^2)}+1=\frac{1+3t^2+t^4}{1+2t^2},\\
&\qquad = 1+t^2-t^4+2t^6-4t^8+8t^{10}+\cdots,\\
&W(t)=\frac{t\sqrt{1+t^2}\left(\frac{(1+t^2)^2}{1+2t^2}-1\right)}{\left(\frac{t}{\sqrt{1+t^2}}\right)^3(1+t^2)}+0=\frac{t^2(1+t^2)}{1+2t^2},\\
&\qquad =t^2-t^4+2t^6-4t^8+8t^{10}+\cdots,
\end{align*}
we have  $\hat d_{n,k}=\hat d_{n-2, k-2}+\hat d_{n-1,k}$, $\hat d_{n,1}=\hat d_{n-1,1}$, and 

\begin{align*}
&\hat d_{n,2}=\hat d_{n-2,0}+\hat d_{n-1,2}-\hat d_{n,4}+\cdots\\
&\hat d_{n,0}=\hat d_{n,2}-\hat d_{n+1,4}+\cdots.
\end{align*}

In the following, we call the compression of a double Riordan array and the compression of a double almost-Riordan array the compressed double Riordan array and the compressed double almost-Riordan array, respectively.

\section{Total positivity of compression double almost-Riordan arrays}

Following Karlin \cite{Kar} and Pinkus \cite{Pin}, an infinite matrix is called totally positive (abbreviate, TP), if its minors of all orders are nonnegative. An infinite nonnegative sequence $(a_n)_{n\geq 0}$ is called a P\'olya frequency sequence (abbreviate, PF), if its Toeplitz matrix

\[
\left[a_{i-j}\right]_{i,j\geq 0}=\left[ \begin{array} {lllll} a_0 & & & &  \\
a_1& a_0 & & & \\ a_2 &a_1& a_0& &  \\ a_3& a_2 & a_1& a_0 & \\
\vdots& \vdots &\vdots&\vdots & \ddots \end{array}\right]
\]
is TP. We say that a finite sequence $(a_0, a_1,\ldots, a_n)$ is PF if the corresponding infinite sequence $(a_0, a_1, \ldots, a_n, 0, \ldots)$ is PF. Denote by ${\bN}$ the set of all nonnegative integers. A fundamental characterization for PF sequences is given by Schoenberg et al.\cite{AESW, ASW, Kar}, which states that a sequence $(a_n)_{n\geq 0}$ is PF if and only if its generating function can be written as 

\begin{equation}\label{0}
\sum_{n\geq 0} a_n t^n=C t^ke^{\gamma t}\frac{\Pi_{j\geq 0} (1+\alpha_j t)}{\Pi_{j\geq 0} (1-\beta_j t)},
\end{equation}
where $C>0$, $k\in {\bN}$, $\alpha_j$, $\beta_j$, $\gamma \geq 0$, and $\sum_{j\geq 0}(\alpha_j+\beta_j)<\infty$. In this case, the above generating function is called a P\'olya frequency formal power series. For some relevant results, see, for example, Brenti \cite{Bre} and Pinkus \cite{Pin}. 

Let $(b|g;f_1,f_2)=(d_{n,k})_{n\geq k\geq 0}$ be a double almost-Riordan array with 

\begin{align*}
& b(t)=\sum_{k\geq 0} b_{2k}t^{2k}, \,\, g(t)=\sum_{k\geq 0} g_{2k}t^{2k},\,\,  f_1(t)=\sum_{k\geq 0}f_{1,2k+1}t^{2k+1}, \,\, \mbox{and} \nonumber\\ 
&  f_2(t)=\sum_{k\geq 0} f_{2,2k+1}t^{2k+1},
\end{align*}
and let its compression array $(\hat d_{n,k})_{n,k\geq 0}$ be defined by \eqref{9.1}, namely, 
$(\hat d_{n,k})_{n,k\geq 0}=(\hat b|\hat g; \hat f_1, \hat f_2)$ with $\hat b$, $\hat g$, $\hat f_1$, and $\hat f_2$ shown in \eqref{9.1-3}. 

It is clear that $(\hat b|\,\hat g;\hat f_1,\hat f_2)$ can be written as 
%\in{\cD}a{\cR}$ can be written as 
\be\label{1.14-2}
(\hat b|\,\hat g;\hat f_1,\hat f_2)=\left( \begin{matrix} \hat b(0) & 0\\ (\hat b-\hat b(0))/t & (\hat g;\hat f_1,\hat f_2) \end{matrix}\right),
\ee
where $(\hat g;\hat f_1,\hat f_2)=(\hat g,\hat g\hat f_1, \hat g\hat f_1\hat f_2,\ldots)$ is a commpression double Riordan array.
%\in{\cD}{\cR}$ is a double Riordan array. 

Let $A$ and $B$ be $m\times m$ and $n\times n$ matrices, respectively. Then we define the direct sum of $A$ and $B$ by

\be\label{2.5-0} 
A\oplus B =\left[ \begin{matrix} A &0 \\ 0 &B\end{matrix}\right]_{(m+n)\times(m+n)}.
\ee

\subsection{Total positivity of the compressions of double Riordan arrays}

We first study the total positivity of the compressions of double Riordan arrays. Then, by using the total positivity of the compressions of double Riordan arrays to study the total positivity of the compressions of double almost-Riordan arrays.

\begin{theorem}\label{thm:6.1} 
Let $(\hat g;\hat f_1,\hat f_2)$ be the compression of a double Riordan array. If $\hat g$, $\hat f_1$, and $\hat f_2$ are P\'olya frequency formal power series, then $(\hat g;\hat f_1,\hat f_2)$ is totally positive.
\end{theorem}

\begin{proof} 
It is well known that a double Riordan array has the decomposition

\be\label{6.1}
(\hat g;\hat f_1,\hat f_2)) = (\hat g;t,t,)((1,\hat f_1,\hat f_2).
\ee
By the classic Cauchy-Binet formula, the product of two TP matrices is also TP. 
So, the total positivity of both $(\hat g;t,t)$ and $(1;\hat f_1,\hat f_2)$ implies total positivity of $(\hat g;\hat f_1,\hat f_2)$. 
If $(\hat g_n=[t^n]g(t))_{n\geq 0}$ is PF, then the corresponding Toeplitz array $(\hat g(t);t,t)$ is TP. Hence, it remains to prove that if the sequence $(\hat f_{i,n}=[t^n] \hat f_i(t))_{n\geq 1}$, $i=1,2$, is PF, then the corresponding Lagrange  array $(1;\hat f_1,\hat f_2)$ is TP. We write 

\be\label{6.2}
(1;\hat f_1,\hat f_2)=\left[ \begin{array}{cc} 1& 0\\ 0&(\hat f_1/t;\hat f_2,\hat f_1)\end{array}\right].
\ee
Hence the total positivity of $(1;\hat f_1,\hat f_2)$ is equivalent to the total positivity of the Bell type double Riordan array $(\hat f_1/t;\hat f_2,\hat f_1)$. Since $(\hat f_1/t;\hat f_2,\hat f_1)=(\hat f_1/t;t,t)(1;\hat f_2,\hat f_1)$, then from \eqref{6.2}

\be\label{6.3}
(\hat f_1/t;\hat f_2,\hat f_1)=(\hat f_1/t;t,t)(1;\hat f_2,\hat f_1)=(\hat f_1/t;t,t)\left[ \begin{array}{cc} 1& 0\\ 0&(\hat f_2/t;\hat f_1,\hat f_2)\end{array}\right].
\ee
Similarly, 

\be\label{6.4}
(\hat f_2/t;\hat f_1,\hat f_2)=(\hat f_2/t;t,t)(1;\hat f_1,\hat f_2)=(\hat f_2/t;t,t)\left[ \begin{array}{cc} 1& 0\\ 0&(\hat f_1/t;\hat f_2,\hat f_1)\end{array}\right].
\ee
So it suffices to prove that if the sequence $(\hat f_{i,n})_{n\geq 1}$, $i=1,2$ are PF, then the corresponding Bell-type array $(\hat f_1/t;\hat f_2,\hat f_1)$ and $(\hat f_2/t;\hat f_1,\hat f_2)$ are TP. 

Denote 

\be\label{6.5}
T_1=(\hat f_1;t,t),\,\, T_2=(\hat f_2;t,t),\,\, H_1=(\hat f_1/t;\hat f_2,\hat f_1),\,\, H_2=(\hat f_2/t;\hat f_1,\hat f_2).
\ee

Let $H_i[n]$, $i=1,2$, be the submatrix consisting of the first $n$ columns of $H_i$, $i=1,2$. Clearly, $H$ is TP if and only if its submatrices $H_i[n]$ are all TP. Thus it suffices to show that $H_i[n]$, $i=1,2$, are TP for all $n\geq 1$. We proceed by induction on $n$. We have by \eqref{6.3} and \eqref{6.4}

\begin{align}
&H_1[n+1]=T_1\left[ \begin{array}{cc} 1&0 \\ 0&H_2[n]\end{array}\right]\quad \mbox{and}\nonumber\\
&H_2[n+1]=T_2\left[ \begin{array}{cc} 1& 0\\ 0&H_1[n]\end{array}\right].\label{6.6}
\end{align}
Clearly, $T_i$, $i=1,2$, are TP since they are the Toeplitz matrices of the PF sequence $(\hat f_{i,n})_{n\geq 1}$. Now, if $H_i[n]$, $i=1,2$, are TP, then so are

\[
\left[ \begin{array}{cc} 1& 0\\ 0&H_i[n]\end{array}\right]
\]
for $i=1,2$. Thus the product $H_i[n + 1]$, $i=1,2$, are also TP by \eqref{6.6}, as desired. This completes the proof. 
\end{proof}

Let us consider two special cases for $\hat b$ in a Riordan array and discuss its TP and the TP for the corresponding almost-Riordan array. 

Based on the discussion in the Introduction, we have

\begin{proposition}\label{pro:2.0}
If the compressed double almost-Riordan array $(\hat b|\,\hat g;\hat f_1,\hat f_2)$ is $TP$, then the corresponding compressed double Riordan array $(\hat g;\hat f_1,\hat f_2)$ is $TP$. 
\end{proposition} 

If $\hat b(t)=1$, then 

\[
(\hat b|\,\hat g;\hat f_1,\hat f_2)=\left( \begin{matrix} 1 & 0\\ 0& (\hat g;\hat f_1,\hat f_2)\end{matrix}\right). 
\]

Hence, one can observe

\begin{proposition}\label{pro:2.2}
The compressed double almost-Riordan array $(1|\,\hat g;\hat f_1,\hat f_2)$ is $TP$ if and only if the compressed double Riordan array $(\hat g;\hat f_1,\hat f_2)$ is $TP$. 
\end{proposition}

The following example shows there exists a non-TP compressed almost-Riordan array $(\hat b|\, \hat g;\hat f_1,\hat f_2)$, in which $\hat b$, $\hat g$, $\hat f_1$, and $\hat f_2$ are P\'olya frequency.

\begin{example}
\label{ex:gf_not_TP}
Let $\hat b(t)=(1+t)^2$, $\hat g(t)=1/(1-t)$, and $\hat f_1(t)=\hat f_2(t)=t$. Then $\hat b(t)$, $\hat g(t)$ and $\hat f_i(t)$ ($i=1,2$) are P\'olya frequency, and $(\hat g;\hat f_1,\hat f_2)$ is TP. However, the minor of $(\hat b|\,\hat g;\hat f_1,\hat f_2))=((1+t)^2|1/(1-t); t, t)$

\[
M^{1,2,3}_{0,1,2}=\det\left( \begin{array}{lll} 2&1&0\\1& 1&1\\0&1&1\end{array}\right)=-1<0.
\]
Hence, $(\hat b|\,\hat g;\hat f_1,\hat f_2)$ is not TP although $(\hat g;\hat f_1,\hat f_2)$ is TP and $b$ is P\'olya frequency. 
\end{example}

\subsection{Total positivity of compressions of double almost-Riordan arrays}

As we have seen, the most interesting question here is whether having fixed $(\hat g;\hat f_1,\hat f_2)$ can we choose $\hat b$ in such way to ensure total positivity of $(\hat b|\hat g;\hat f_1,\hat f_2)$ and what, in such case, is the connection between $\hat b$ and the triple $(\hat g,\hat f_1,\hat f_2)$.  

For every double Riordan array $(\hat g;\hat f_1,\hat f_2)$ there exists $\hat b$ such that the double almost-Riordan array $(\hat b|\, \hat g;\hat f_1,\hat f_2)$ is TP. In fact, we have 

\begin{theorem}\label{thm:2.4-2}
Suppose even $\hat g\in\F_0$, odd $\hat f_1,\hat f_2\in\F_1$ such that the compressed double Riordan array $(\hat g;\hat f_1,\hat f_2)$ is TP. Then the compressed double almost-Riordan array $\left(t\cdot g+\alpha|\,\hat g;\hat f_1,\hat f_2\right)$ is also TP for every $\alpha>0$.
\end{theorem}

\begin{proof}
Observe that the compressed double almost-Riordan array $\left(t\cdot \hat g+\alpha|\,\hat g;\hat f_1,\hat f_2\right)=R$ can be written in the partitioned form
\[
R=\left[
\begin{array}{c|c}
\alpha & 0 \\
\hline 
\hat g & (\hat g;\hat f_1,\hat f_2) \\
\end{array}
\right].
\]
We now consider its minors.
\begin{enumerate}
\item 
If a minor of $R$ does not contain the column $0$, then it is simply a minor of $(\hat g;\hat f_1,\hat f_2)$, so it is nonnegative.
\item 
Suppose now that a minor of $R$ contains the column $0$.
\begin{enumerate}
\item 
If it contains the row $0$, then expanding along it, we get a minor of $(\hat g;\hat f_1,\hat f_2)$ multiplied by $\alpha$ which is a nonnegative number.
\item 
If this minor does not contain the row $0$, then it is a minor of the partitioned matrix 
\[
\tilde R=\left[\begin{array}{c|c|c}
\hat g_0 & g_0 & 0 \\
\hline
\frac{\hat g-\hat g_0}t & \frac{\hat g-\hat g_0}t & \left(\frac{\hat g\hat f_1}t;\hat f_1,\hat f_2\right)\\
\end{array}\right].
\]
\begin{enumerate}
\item 
Now, if minor contains both the two first columns of $\tilde R$, then it is equal to $0$.  
\item 
If the minor contains one or none of the first two columns of $\tilde R$, then it is simply one of the minors of $(\hat g,\hat f)$, so it is nonnegative. 
\end{enumerate}
\end{enumerate}
\end{enumerate}
\end{proof}

\par 
Moreover, for every TP compressed double Riordan array $(\hat g;\hat f_1,\hat f_2)$ there exist $\hat b$ that does not depend on $\hat g$, $\hat f_1$, and $\hat f_2$ such that $(\hat b|\, \hat g;\hat f_1,\hat f_2)$ is TP.

\begin{theorem}\label{thm:2.4-3}
Let even $\hat g\in\F_0$, odd $\hat f_1,\hat f_2\in\F_1$ and let $\hat b=\hat b_0+\hat b_1t$ with $\hat b_0,\hat b_1\geqslant 0$. If the compression double Riordan array $(\hat g;\hat f_1,\hat f_2)$ is TP, then $(\hat b_0+\hat b_1t|\,\hat g;\hat f_1,\hat f_2)$ is also TP.
\end{theorem}

\begin{proof}
Let 
\[
R=(\hat b_0+\hat b_1t|\,\hat g;\hat f_1,\hat f_2)=\left[
\begin{array}{c|c}
\hat b_0 & 0 \\
\hline 
\hat b_1 &  \\
0 & (\hat g;\hat f_1,\hat f_2) \\
\vdots &  \\
\end{array}
\right].
\]
Consider the minors of $R$.
\begin{enumerate}
\item If a minor does not contain the column $0$, then it is a minor of $(\hat g;\hat f_1,\hat f_2)$ that, by the assumption, is nonnegative. 
\item Suppose now that a minor contains the column $0$. 
\begin{enumerate}
\item 
If it contains the row $0$, then, after expansion along this row, one gets a minor of $(\hat g;\hat f_1,\hat f_2)$ multiplied by $b_0$ which is a nonnegative number.
\item 
If it does not contain the row $0$, but it contains the first row of $R$, then it is a minor of 
\[
\tilde R=\left[\begin{array}{c|c|c}
\hat b_1 & \hat g_0 & 0 \\
\hline
0 & \frac{\hat g-\hat g_0}t & \left(\frac{\hat g\hat f_1}t;\hat f_1,\hat f_2\right)\\
\end{array}\right].
\]
Thus, expanding now along the column $0$, we get a minor of $(\hat g;\hat f_1,\hat f_2)$ multiplied by $b_1$ that is nonnegative.
\item 
If it does not contain neither the zeroth nor the first row of $R$ then it is a minor of 
\[
\tilde R=\left[\begin{array}{c|c|c}
0 & \frac{\hat g-\hat g_0}t & \left(\frac{\hat g\hat f_1}t;\hat f_1,\hat f_2\right)\\
\end{array}\right]
\]
and it is equal to $0$.
\end{enumerate}
\end{enumerate}
\end{proof}

\medbreak

\end{document}